\documentclass[11pt,twoside]{article}
\usepackage{amsmath,amssymb,graphicx,enumitem,mathtools,subcaption,comment,multirow,blkarray,booktabs,bm}
\usepackage{amsthm,enumitem}
\usepackage{url}
\usepackage{breakurl}

\usepackage[colorlinks=true,citecolor=blue,linkcolor=blue,urlcolor=blue,bookmarks,bookmarksopen,bookmarksdepth=2,backref=page,breaklinks]{hyperref}

\usepackage{bookmark}
\bookmarksetup{
	numbered, 
	open,
}

\usepackage{listings}
\usepackage{fancyvrb}
\usepackage{soul}

\usepackage[a4paper]{geometry}
\makeatletter
\renewcommand\title[1]{\gdef\@title{\reset@font\Large\bfseries #1}}
\renewcommand\section{\@startsection {section}{1}{\z@}%
	{-3.5ex \@plus -1ex \@minus -.2ex}%
	{2.3ex \@plus.2ex}%
	{\normalfont\large\bfseries}}
\renewcommand\subsection{\@startsection{subsection}{2}{\z@}%
	{-3ex\@plus -1ex \@minus -.2ex}%
	{1.5ex \@plus .2ex}%
	{\normalfont\normalsize\bfseries}}
\renewcommand\subsubsection{\@startsection{subsubsection}{3}{\z@}%
	{-2.5ex\@plus -1ex \@minus -.2ex}%
	{1.5ex \@plus .2ex}%
	{\normalfont\normalsize\bfseries}}

\def\@runningauthor{}\newcommand{\runningauthor}[1]{\def\runningauthor{#1}}
\def\@runningtitle{}\newcommand{\runningtitle}[1]{\def\runningtitle{#1}}

\renewcommand{\ps@plain}{%
	\renewcommand{\@evenhead}{\footnotesize\scshape \hfill\runningauthor\hfill}
	\renewcommand{\@oddhead}{\footnotesize\scshape \hfill\runningtitle\hfill}}
\g@addto@macro\bfseries{\boldmath}

\makeatother

\newcommand\blfootnote[1]{%
	\begingroup
	\renewcommand\thefootnote{}\footnote{#1}%
	\addtocounter{footnote}{-1}%
	\endgroup
}

\let\OLDthebibliography\thebibliography
\renewcommand\thebibliography[1]{
	\OLDthebibliography{#1}
	\setlength{\parskip}{0pt}
	\setlength{\itemsep}{0pt plus 0.3ex}
}

\setlist{leftmargin=*, itemsep=2pt, topsep=2pt, parsep=0pt, partopsep=5pt}

\newcommand*{\doi}[1]{\href{\detokenize{#1}}{doi: \detokenize{#1}}}
\renewcommand*{\backref}[1]{}
\renewcommand*{\backrefalt}[4]{%
	\ifcase #1 (Not cited.)%
	\or        (Cited on page~#2.)%
	\else      (Cited on pages~#2.)%
	\fi}

\DeclareMathOperator{\snf}{SNF}
\DeclareMathOperator{\gjb}{GJB}

\DeclareMathOperator{\dev}{dev}
\DeclareMathOperator{\ind}{ind}

\DeclareMathOperator{\rind}{r-ind}
\DeclareMathOperator{\spa}{span}

\DeclareMathOperator{\rank}{rank}

\DeclareMathOperator{\tworank}{2-rank}
\DeclareMathOperator{\Grank}{\Gamma-rank}

\theoremstyle{plain}
\newtheorem{theorem}{Theorem}[section]

\newtheorem{proposition}[theorem]{Proposition}
\newtheorem{result}[theorem]{Result}

\newtheorem{corollary}[theorem]{Corollary}

\theoremstyle{definition}
\newtheorem{conjecture}[theorem]{Conjecture}
\newtheorem{example}[theorem]{Example}
\newtheorem{openproblem}[theorem]{Open Problem}

\newtheorem{definition}[theorem]{Definition}
\newtheorem{remark}[theorem]{Remark}
\newtheorem{observation}[theorem]{Observation}

\numberwithin{theorem}{section}
\numberwithin{equation}{section}
\numberwithin{table}{section}

\usepackage[labelsep=period]{caption}

\usepackage{algpseudocode}
\usepackage[Algorithmus,section]{algorithm}
\algnewcommand{\IfOneRow}[1]{\State\algorithmicif\ #1,}
\algnewcommand{\EndifOneRow}{}

\algdef{SE}[SUBALG]{Indent}{EndIndent}{}{\algorithmicend\ }%
\algtext*{Indent}
\algtext*{EndIndent}

\makeatletter
\renewcommand{\ALG@name}{Algorithm}
\makeatother

\newcommand{\F}{\mathbb F}
\newcommand{\Z}{\mathbb Z}
\newcommand{\N}{\mathbb N}

\newcommand{\MClass}{\mathcal{M}^{\#}}
\newcommand{\MrsClass}{\mathcal{M}_{r,s}^{\#}}

\title{Cubic bent functions outside the completed Maiorana-McFarland class}
\runningtitle{Cubic bent functions outside the completed Maiorana-McFarland class}

\author{
	Alexandr A. Polujan
	\qquad  Alexander Pott\\
	\small Faculty of Mathematics\\[-0.8ex]
	\small Institute of Algebra and Geometry\\[-0.8ex]
	\small Otto von Guericke University \\[-0.8ex]
	\small Universit\"{a}tsplatz 2, 39106, Magdeburg, Germany\\
	\small\tt alexandr.polujan@$\{$gmail.com,ovgu.de$\}$, alexander.pott@ovgu.de
}

\runningauthor{A.\ A.\ Polujan, A.\ Pott}

\makeatletter
\let\Date\@date
\makeatother

\begin{document}
	
\maketitle
	
\thispagestyle{empty}
	
\begin{abstract}
\noindent In this paper we prove that in opposite to the cases of 6 and 8 variables, the Maiorana-McFarland construction does not describe the whole class of cubic bent functions in $n$ variables for all $n\ge 10$. Moreover, we show that for almost all values of $n$, these functions can simultaneously be homogeneous and have no affine derivatives. \blfootnote{The first version of this work~\cite{PolujanPott2019} was presented in the ``Eleventh	International Workshop on Coding and Cryptography (WCC 2019)''.}


\end{abstract}

\section{Introduction}
Bent functions, introduced by Rothaus in~\cite{ROTHAUS1976300}, are Boolean functions having the maximum Hamming distance from the set of all affine functions. Being extremal combinatorial objects, they have been intensively studied in the last four decades, due to their broad applications to cryptography, coding theory and theory of difference sets. 

Cubic bent functions, i.e. bent functions of algebraic degree three, attracted a lot of attention from researchers, partly because  small algebraic degree of these functions allows to investigate them exhaustively, when the number of variables is not too large. For instance, all cubic bent functions in six and eight variables are well-understood: the classification is given in~\cite{Braeken2006,ROTHAUS1976300}, the enumeration was obtained in~\cite{LangevinL11,Preneel1993}, and all these functions belong to the completed Maiorana-McFarland class $\MClass$~\cite{Braeken2006,Dillon1972}. A couple of infinite families of cubic bent functions were constructed recently, however, some of them~\cite{CanteautCK08,Leander06} are proved to be the members of $\MClass$, while some of them are not analyzed yet~\cite{DobbertinLCCFG06,Mesnager2014NewCubicBents}. Therefore, it is not clear, whether an $n$-variable cubic bent function can be outside the $\MClass$ class whenever $n\ge10$. At the same time, cubic bent functions, which are homogeneous or have no affine derivatives, are of a special interest.

A cubic function has no affine derivatives, if all its non-trivial first-order derivatives are quadratic, what makes cryptographic systems with such components more resistant to certain differential attacks. It is well-known that cubic bent functions without affine derivatives exist for all even $n\ge6,n\neq8$, as it was shown in in~\cite{Canteaut2003DecomposingBF,HOU1998149}. Recently Mandal, Gangopadhyay and St\u{a}nic\u{a} in~\cite{MandalCubicNoAffDers} constructed two classes of cubic bent functions without affine derivatives inside $\MClass$ and proved their mutual inequivalence. They also suggested to find such functions outside the $\MClass$ class and evaluate their significance for cryptographic applications~\cite[Section 1.6]{MandalCubicNoAffDers}.

A Boolean function is called homogeneous, if all the monomials in its algebraic normal form have the same algebraic degree.  Homogeneous cubic bent functions were firstly considered by Qu, Seberry and Pieprzyk in~\cite{QuSP99}, motivated by faster evaluation in cryptographic systems. The only known homogeneous bent functions are quadratic and cubic, moreover, it is not known, whether a homogeneous bent functions of higher degrees exist. While the characterization of homogeneous quadratic bent functions is well-known~\cite[Chapter 15]{MacWilliamsSloane}, it is in general a difficult task to construct a homogeneous cubic bent function. The only known primary construction was given by Seberry, Xia and Pieprzyk in~\cite{SeberryConstruction}. They proved, that a proper linear transformation of variables can bring special non-homogeneous cubic bent function from $\MClass$ to a homogeneous one. Unfortunately, all functions of this type have many affine derivatives. Another approach is based on the concatenation of homogeneous cubic bent functions in a small number of variables via direct sum. The known computational construction methods of such functions include:
\begin{itemize}
	\item The tools from the modular invariant theory, as it was shown  by Charnes, R\"otteler and Beth in~\cite{Charnes2002};
	\item The significant reduction of the search space, suggested by Meng et al. in~\cite{MengNovel2004}.
\end{itemize}
Using these approaches, the mentioned authors constructed a lot of homogeneous cubic bent functions in a small number of variables $6\le n\le12$. However, since all these examples have not been analyzed with respect to being outside the $\MClass$ class and having no affine derivatives, it is not clear, which properties can the concatenations of these functions have.

The aim of this paper is two-fold. First, we analyze the known homogeneous cubic bent functions in ten and twelve variables from~\cite{Charnes2002,MengNovel2004} and show, that some of these functions do not belong to the the $\MClass$ class and all of them are different from the primary construction of Seberry, Xia and Pieprzyk~\cite{SeberryConstruction}. Moreover, some of them have no affine derivatives. Secondly, we extend these results for infinite families, by showing, that proper direct sums of these functions inherit the properties of its summands. Consequently, we prove that for any $n\ge8$ there exist cubic bent functions inside $\MClass$, but different from the primary construction~\cite{SeberryConstruction}. Further, we consider cubic bent functions with respect to the following three properties: outside $\MClass$, without affine derivatives, and homogeneous. We show, that $n$-variable cubic bent functions with at least two of the three mentioned properties exist for all $n\ge n_0$, where $n_0$ depends on the selected combination of properties. In this way, we prove that in general the whole class of cubic bent functions in $n$ variables is not described by the $\MClass$ class, whenever $n\ge10$. Finally, we show existence of cubic bent functions without affine derivatives outside $\MClass$, thus solving a recent open problem by Mandal, Gangopadhyay and St\u{a}nic\u{a}~\cite[Section 1.6]{MandalCubicNoAffDers}.

The paper is organized in the following way. In Subsection~\ref{subsection: Preliminaries} we introduce some basic notions and background on Boolean functions. Section~\ref{section: Geometric Invariants of Boolean functions} describes geometric invariants of Boolean functions, which we use in the next section in order to distinguish inequivalent functions. Section~\ref{section: Homogeneous cubic bent functions} deals with the construction of new homogeneous cubic bent functions from old. First, in Subsection~\ref{subsection: The known examples and constructions} we survey the known homogeneous bent functions, provide the classification of known examples and show, that some of them are not in the $\MClass$ class. In Subsection~\ref{subsection: Homogeneous cubic bent functions, different from the primary construction}, we show that proper concatenations of homogeneous cubic functions can never be equivalent to the primary construction. Finally, in Subsection~\ref{subsection: New from old} we introduce an approach, aimed to produce many homogeneous functions from a single given one without increasing the number of variables, and illustrate its application for homogeneous cubic bent functions in 12 variables. Section~\ref{section: Bent functions outside M} deals with the construction of cubic bent functions outside the $\MClass$ class, using the direct sum. In Subsection~\ref{subsection: MSubspaces} we provide a sufficient condition, explaining how one should select bent functions $f$ and $g$, such that the direct sum $f\oplus g$ is outside $\MClass$. In Subsection~\ref{subsection: Application} we show, that certain cubic bent functions in $6\le n\le12$ variables satisfy our new sufficient condition and thus lead to infinitely many cubic bent functions outside the $\MClass$ class, which are homogeneous or do not have affine derivatives. The paper is concluded in Section~\ref{section: Conclusion} and cubic bent functions, used in the paper, are given in the Appendix~\ref{section: Appendix}.

\subsection{Preliminaries}\label{subsection: Preliminaries}
Let $\F_2=\{0,1\}$ be the finite field with two elements and let $\F_2^n$ be the vector space of dimension $n$ over $\F_2$. Mappings $f\colon \F_2^n\rightarrow\F_2$ are called \emph{Boolean functions} in $n$ variables. A Boolean function on $\F_2^n$ can be uniquely expressed as a multivariate polynomial in the ring $\F_2[x_1,\dots,x_n]/(x_1\oplus x_1^2,\dots,x_n\oplus x_n^2)$. This representation is unique and called the \emph{algebraic normal form} (denoted further as ANF), that is,
$$f(\mathbf{x})=\bigoplus\limits_{\mathbf{v}\in\F_2^n}c_{\mathbf{v}} \left( \prod_{i=1}^{n} x_i^{v_i} \right),$$
where $\mathbf{x} = (x_1,\dots, x_n)\in\F_2^n$,  $c_{\mathbf{v}}\in\F_2$ and $\mathbf{v} = (v_1,\dots, v_n)\in\F_2^n$. The \textit{complement} of a Boolean function $f$ is defined by $\bar{f}:= f\oplus 1$. The \emph{algebraic degree} of a Boolean function $f$, denoted by $\deg(f)$, is the algebraic degree of its ANF. We call a Boolean function $d$\emph{-homogeneous}, if all the monomials in its ANF have the same degree $d$, and simply \emph{homogeneous}, if the degree is clear from the context.

With a Boolean function $f\colon\F_2^n\rightarrow\F_2$ one can associate the mapping $D_{\mathbf{a}}f(\mathbf{x}):=f(\mathbf{x} \oplus \mathbf{a})\oplus f(\mathbf{x})$, which is called the \textit{first-order derivative} of a function $f$ in the \emph{direction} $\mathbf{a}\in\F_2^n$. Derivatives of higher orders are defined recursively, i.e. the \emph{$k$-th order derivative} of a function $f$ is given by $D_{\mathbf{a}_k}D_{\mathbf{a}_{k-1}}\ldots D_{\mathbf{a}_1}f(\mathbf{x}):=D_{\mathbf{a}_k}(D_{\mathbf{a}_{k-1}}\ldots D_{\mathbf{a}_1}f)(\mathbf{x})$. For instance, the \emph{second-order derivative} of $f$ is given by $D_{\mathbf{a},\mathbf{b}}f(\mathbf{x}):=D_{\mathbf{b}}(D_{\mathbf{a}}f)(\mathbf{x})=f(\mathbf{x} \oplus \mathbf{a} \oplus \mathbf{b})\oplus f(\mathbf{x} \oplus \mathbf{a}) \oplus f(\mathbf{x} \oplus \mathbf{b}) \oplus f(\mathbf{x})$. The point $\mathbf{a}\in\F_2^n$ is called a \emph{fast point} of a function $f\colon\F_2^n\rightarrow\F_2$ if it satisfies $\deg(D_\mathbf{a}f) < \deg(f)-1$ and a \emph{slow point}, if $\deg(D_\mathbf{a}f)=\deg(f)-1$. The set of fast points $\mathbb{FP}_f$ forms a vector subspace and its dimension is bounded by $\dim(\mathbb{FP}_f) \le n-\deg(f)$, as it was shown in~\cite{DuanLai10}. A cubic function has \emph{no affine derivatives}, if $\dim(\mathbb{FP}_f)=0$, i.e. all its non-trivial first-order derivatives are quadratic functions. 

The \emph{direct sum} of two functions $f\colon\F_2^n\rightarrow\F_2$ and $g\colon\F_2^m\rightarrow\F_2$ is a function $h\colon\F_2^{n+m}\rightarrow\F_2$, defined by $h(\mathbf{x},\mathbf{y}):=f(\mathbf{x})\oplus g(\mathbf{y})$. We also define the \emph{$k$-fold direct sum} $k\cdot f\colon \F_2^{k\cdot n}\rightarrow\F_2$ as $k\cdot f(\mathbf{x}_1,\ldots,\mathbf{x}_k):=f(\mathbf{x}_1)\oplus\dots\oplus f(\mathbf{x}_k),\mbox{ for }\mathbf{x}_i \in \F_2^n.$
\begin{definition}
	A Boolean function $f\colon\F_2^n\rightarrow\F_2$ is called \emph{bent}, if for all $\mathbf{a}\in\F_2^n$ with $\mathbf{a}\neq\mathbf{0}$ and all $b\in\F_2$ the equation $D_{\mathbf{a}}f(\mathbf{x})=b$ has $2^{n-1}$ solutions $\mathbf{x}\in\F_2^n$.
\end{definition}
\begin{remark}
	It is well-known, that bent functions in $n$ variables exist only for $n$ even and have degree at most $n/2$ (see~\cite{ROTHAUS1976300}).
\end{remark}	
On the set of all Boolean functions one can introduce an equivalence relation in the following way: two functions $f,f'\colon\F_2^n\rightarrow\F_2$ are called \emph{equivalent}, if there exists a non-degenerate affine transformation $A\in AGL(n,2)$ and an affine function $l(\mathbf{x})=\langle \mathbf{a},\mathbf{x} \rangle_n\oplus b$ on $\F_2^n$ (where $\mathbf{x}\in\F_2^n$, $b\in\F_2$ and $\langle \cdot,\cdot \rangle_n$ is a non-degenerate bilinear form on $\F_2^n$), such that  $f'(\mathbf{x})=f(\mathbf{x}A)\oplus l(\mathbf{x})$ holds for all $\mathbf{x}\in\F_2^n$.

Further we will analyze inequivalence of Boolean functions with the help of incidence structures and linear codes. Recall that an \emph{incidence structure} is a triple $\mathbb{S}=(\mathcal{P}, \mathcal{B}, \mathcal{I}),$ where $\mathcal{P}=\{p_1,\ldots,p_v\}$ is a set of elements called \emph{points} and $\mathcal{B}=\{ B_1,\ldots,B_b \}$ is a set of elements called \emph{lines}, and $\mathcal{I} \subseteq \mathcal{P} \times \mathcal{B}$ is a binary relation, called \emph{incidence relation}. The \emph{incidence matrix} of $M(\mathbb{S})=(m_{ij})$ of $\mathbb{S}$ is a binary $b\times v$ matrix with $m_{ij}=1$ if $p_j\in B_i$ and $m_{ij}=0$ otherwise. Two incidence structures $\mathbb{S}$ and $\mathbb{S}'$ are \emph{isomorphic}, if there are permutation matrices $P$ and $Q$ such that $P \cdot M(\mathbb{S}) \cdot Q=M(\mathbb{S}')$.

The \emph{linear code} of $\mathbb{S}$ over $\F_2$ is the subspace $\mathcal{C}(\mathbb{S})$ of $\F_2^v$, spanned by the row vectors of the incidence matrix $M(\mathbb{S})$. It is clear, that the incidence matrix $M(\mathbb{S})$ and the linear code $\mathcal{C}(\mathbb{S})$ depend on the labeling of the points and lines of $\mathbb{S}$, however these objects are essentially unique up to row and column permutations. We refer to~\cite{Ding:CodesfromDS,Ding:DesignsFromCodes} about incidence structures and their linear codes. 

Finally, we will use the following notation for vectors and matrices: $\mathbf{j}_n$ is the \emph{all-one-vector} of length $n$, by $\mathbf{I}_n$ and $\mathbf{J}_n$ we denote the \emph{identity matrix} and the \emph{all-one-matrix} of order $n$. The \emph{all-zero-matrix} of order $n$ and size $r\times s$ is denoted by $\mathbf{O}_{n}$ and $\mathbf{O}_{r,s}$ respectively.
\subsection{The completed generalized Maiorana-McFarland class of Boolean functions}\label{subsubsection: Classification and Maiorana-McFarland Test}
The \emph{generalized Maiorana-McFarland class} $\mathcal{M}_{r,s}$ of Boolean functions in~$n=r+s$ variables~\cite[p. 354]{carlet_2010} is the set of Boolean functions of the form
\begin{equation}\label{equation: Maiorana-McFarland Representation}
f_{\pi,\phi}(\mathbf{x}, \mathbf{y}) = \langle \mathbf{x},\pi(\mathbf{y}) \rangle_r\oplus \phi(\mathbf{y}),
\end{equation}
where $\mathbf{x}\in\F_2^r,\ \mathbf{y}\in\F_2^s$, $\phi$ is an arbitrary Boolean function on $\F_2^s$ and $\pi\colon\F_2^s\rightarrow\F_2^r$ is some mapping. A function $f$ belongs to the \emph{completed generalized Maiorana-McFarland class} $\MrsClass$, if it is equivalent to some function from $\mathcal{M}_{r,s}$. In the case $r=s$, which corresponds to the \emph{original Maiorana-McFarland class} of bent functions $\mathcal{M}$, a function $f$ is bent if and only if the mapping $\pi$ is a permutation \cite[p. 325]{carlet_2010}. The completed version of $\mathcal{M}$ is denoted by $\MClass$. We will call~\eqref{equation: Maiorana-McFarland Representation} a \emph{Maiorana-McFarland representation} of a given function $f$ on $\F_2^n$, if there exists a non-degenerate linear transformation $A$, s.t. $f(\mathbf{z}A)=f_{\pi,\phi}(\mathbf{x}, \mathbf{y})$ for some mappings $\pi$ and $\phi$.

A characterization of the completed Maiorana-McFarland class $\MClass$ of bent functions is given in~\cite[p. 102]{Dillon1974} and~\cite[Lemma 33]{CanteautDDL06}. In the case of the $\MrsClass$ class, the proof is similar.
\begin{proposition}\label{proposition: MM} Let $f$ be a Boolean function on $\F_2^n$ with $n=r+s$. The following statements are equivalent:
	\begin{enumerate}
		\item The function $f$ belongs to the $\MrsClass$ class.
		\item There exists a vector subspace $U$ of dimension $r$ such that the second order derivatives $D_{\mathbf{a},\mathbf{b}}f$ vanish for all $\mathbf{a},\mathbf{b}\in U$, that means $D_{\mathbf{a},\mathbf{b}}f=0$.
		\item\label{proposition: MM part (iii)} There exists a vector subspace $U$ of dimension $r$ such that the function $f$ is affine on every coset of $U$.
	\end{enumerate}
\end{proposition}
Motivated by this characterization, we introduce $\mathcal{M}$-subspaces of Boolean functions, as those, which satisfy the second statement of the Proposition~\ref{proposition: MM}. 
\begin{definition}\label{definition: MSubspace}
	We will call a vector subspace $U$ an \emph{$\mathcal{M}$-subspace} of a Boolean function $f\colon\F_2^n\rightarrow\F_2$, if for all $\mathbf{a},\mathbf{b}\in U$ the second-order derivatives $D_{\mathbf{a},\mathbf{b}}f$ are constant zero functions, i.e $D_{\mathbf{a},\mathbf{b}}f=0$. We denote by $\mathcal{MS}_r(f)$ the collection of all $r$-dimensional $\mathcal{M}$-subspaces of $f$ and by $\mathcal{MS}(f)$ the collection 
	$$\mathcal{MS}(f):=\bigcup\limits_{r=1}^{n} \mathcal{MS}_r(f).$$
\end{definition}
The following invariant, called linearity index~\cite[p. 82]{Yas97}, measures the maximal possible number of variables of linear functions in a  Maiorana-McFarland representation~\eqref{equation: Maiorana-McFarland Representation} of a Boolean function.
\begin{definition}\label{definition: Linearity Index}
	The \emph{linearity index} $\ind(f)$ of a Boolean function $f\colon\F_2^n\rightarrow\F_2$ is the maximal possible $r$, such that $f\in\MrsClass$. In terms of $\mathcal{M}$-subspaces, the linearity index of $f$ is given by $\ind(f)=\max\limits_{U\in \mathcal{MS}(f)}\dim(U)$.
\end{definition}
\begin{example}\label{example: M-subspace}
	Let $f(\mathbf{x}):=x_1 x_4\oplus x_2 x_5\oplus x_3 x_6\oplus x_1 x_2 x_3$ be a cubic Maiorana-McFarland bent function on $\F_2^6$. Second-order derivatives of $f$ are given by the function $D_{\mathbf{a},\mathbf{b}}f(\mathbf{x})=c_{\mathbf{0}}(\mathbf{a},\mathbf{b})\oplus (a_3 b_2 \oplus a_2 b_3) x_1 \oplus (a_3 b_1 \oplus a_1 b_3 )x_2 \oplus (a_2 b_1 \oplus a_1 b_2 )x_3$, where the constant term $c_{\mathbf{0}}(\mathbf{a},\mathbf{b})$ depends on $\mathbf{a},\mathbf{b}$ and is given by $c_{\mathbf{0}}(\mathbf{a},\mathbf{b}):=a_1 (a_2 b_3 \oplus a_3 b_2 \oplus b_2 b_3) \oplus b_1(a_2 a_3 \oplus a_2 b_3 \oplus a_3 b_2) \oplus a_1 b_4 \oplus a_2 b_5 \oplus a_3 b_6 \oplus a_4 b_1 \oplus a_5 b_2 \oplus a_6 b_3$.
	One can check that the subspace $U=\scalebox{1}{$\langle (0, 0, 0, 1, 0, 0), (0, 0, 0, 0, 1, 0), (0, 0, 0, 0, 0, 1)\rangle$}$ is an $\mathcal{M}$-subspace of $f$, since its second-order derivatives $D_{\mathbf{a},\mathbf{b}}f$, which correspond to all two-dimensional vector subspaces $\langle\mathbf{a},\mathbf{b}\rangle $ of $U$, are constant zero functions
	$$
	\scalebox{0.83}{$\begin{gathered}
	\scalebox{1}{$\left\langle
		\begin{array}{cccccc}
		0 & 0 & 0 & 0 & 1 & 0 \\
		0 & 0 & 0 & 0 & 0 & 1 \\
		\end{array}
		\right\rangle\mapsto 0$},
	\scalebox{1}{$\left\langle
		\begin{array}{cccccc}
		0 & 0 & 0 & 1 & 0 & 0 \\
		0 & 0 & 0 & 0 & 0 & 1 \\
		\end{array}
		\right\rangle\mapsto 0$},
	\scalebox{1}{$\left\langle
		\begin{array}{cccccc}
		0 & 0 & 0 & 1 & 1 & 0 \\
		0 & 0 & 0 & 0 & 0 & 1 \\
		\end{array}
		\right\rangle\mapsto 0$},
	\scalebox{1}{$\left\langle
		\begin{array}{cccccc}
		0 & 0 & 0 & 1 & 0 & 0 \\
		0 & 0 & 0 & 0 & 1 & 0 \\
		\end{array}
		\right\rangle\mapsto 0$},
	\\
	\scalebox{1}{$\left\langle
		\begin{array}{cccccc}
		0 & 0 & 0 & 1 & 0 & 1 \\
		0 & 0 & 0 & 0 & 1 & 0 \\
		\end{array}
		\right\rangle\mapsto 0$},
	\scalebox{1}{$\left\langle
		\begin{array}{cccccc}
		0 & 0 & 0 & 1 & 0 & 0 \\
		0 & 0 & 0 & 0 & 1 & 1 \\
		\end{array}
		\right\rangle\mapsto 0$},
	\scalebox{1}{$\left\langle
		\begin{array}{cccccc}
		0 & 0 & 0 & 1 & 0 & 1 \\
		0 & 0 & 0 & 0 & 1 & 1 \\
		\end{array}
		\right\rangle\mapsto 0$}.
	\end{gathered}$}
	$$
\end{example}
Now we describe a naive algorithm, which one can use to construct the collection $\mathcal{MS}_r(f)$ for a given function $f$ and a fixed $r$. For a more efficient algorithm we refer to~\cite[Algorithm 2]{CanteautDDL06}.
\begin{algorithm}[H]
	\caption{Construct the collection $\mathcal{MS}_r(f)$.}
	\label{algorithm: f in M}
	\begin{algorithmic}[1]
		\Require A Boolean function $D_{\mathbf{a},\mathbf{b}}f:\F_2^{n}\rightarrow\F_2$ and $2 \le r \le n$.
		\Ensure The collection $\mathcal{MS}_r(f)$.	
		\State\textbf{Construct} $\mathcal{MS}_2(f):=\{ \langle \mathbf{a},\mathbf{b} \rangle: \dim(U)=2 \mbox{ and }D_{\mathbf{a},\mathbf{b}}f=0 \}$.
		\ForAll{subspaces $U\in \mathcal{MS}_2(f)$}
		\Repeat  
		\State \textbf{Determine} subspaces $\tilde{U}=\langle U,\tilde{\mathbf{u}} \rangle$ for all $\tilde{\mathbf{u}}\notin U$, such that for any two-dimensional \Indent vector subspace $ \langle\mathbf{a},\mathbf{b}\rangle\subseteq U$ second-order derivatives $D_{\mathbf{a},\mathbf{b}}f=0$. \EndIndent
		\State \textbf{Put} $U\gets \tilde{U}$ for the obtained subspaces $\tilde{U}$. 		
		\Until{$\dim(U)=r$.}		
		\State \textbf{Output} subspaces $U$ of dimension $r$.
		\EndFor
	\end{algorithmic}
\end{algorithm}
\begin{remark}\label{remark: Compute ind(f)}
	Algorithm~\ref{algorithm: f in M} can be used to compute the linearity index of a given function $f$ in the following way: $\ind(f)$ is the biggest $r$, for which $\mathcal{MS}_r(f)\ne\varnothing$.
\end{remark}

\begin{remark}\label{remark: How to construct a linear mapping}
	For a given $\mathcal{M}$-subspace $U\in\mathcal{MS}_r(f)$ of a function $f\colon\F_2^n\rightarrow\F_2$ one can construct an invertible matrix $A_U$, which brings $f$ to its Maiorana-McFarland representation~\eqref{equation: Maiorana-McFarland Representation}, i.e. $f(\mathbf{z}A_U)=\langle \mathbf{x} , \pi(\mathbf{y})  \rangle_r \oplus \phi(\mathbf{y})$, with $\mathbf{z}\in\F_2^n$, $\mathbf{x}\in\F_2^r$ and $\mathbf{y}\in\F_2^s$, in the following way: since the values of $\langle \mathbf{x} , \pi(\mathbf{y})  \rangle_r \oplus \phi(\mathbf{y})$ on the coset $\F_2^r\oplus \mathbf{y}$ for $\mathbf{y}\in\F_2^s$ coincide with the values of $f$ on the coset $U\oplus \bar{\mathbf{u}}$ for $\bar{\mathbf{u}}\in \bar{U}$, we can construct $A_U$ using the change of basis formula
	\begin{equation}\label{equation: Linear Transform}
	A_U=
	\left(
	\begin{array}{c|c}
	\mathbf{O}_{r,s} & \mathbf{I}_{r} \\
	\hline
	\mathbf{I}_{s} & \mathbf{O}_{s,r} 
	\end{array}
	\right)
	\cdot
	\left(
	\begin{array}{c}
	\gjb(\bar{U}) \\
	\hline
	\gjb(U) 
	\end{array}
	\right).
	\end{equation}
	Here $\gjb(U)$ denotes the \emph{Gauss-Jordan basis} of a vector space $U$ and $\bar{U}$ is the \emph{complement} of $U$, i.e. $\dim(U)+\dim(\bar{U})=n$ and $U \cap \bar{U}=\{ \mathbf{0} \}$, which we compute as in~\cite[Subsection 4]{CanteautDDL06}.
\end{remark}

\section{Geometric invariants of Boolean functions}\label{section: Geometric Invariants of Boolean functions}
In this section we study invariants of Boolean functions, which arise from certain binary matrices. We call these invariants \emph{geometric}, since any $(0,1)$-matrix defines an incidence structure, and hence a finite geometry, and will use them in the next section to distinguish inequivalent homogeneous cubic bent functions.
\subsection{Incidence structures from Boolean functions}\label{subsection: Geometric Invariants}	
For a subset $A$ of an additive group $(G,+)$ the \emph{development} $\dev(A)$ of $A$ is an incidence structure, whose points are the elements in $G$, and whose lines are the translates $A + g := \{a+g:a\in A \}$. For a Boolean function $f\colon\F_2^n\rightarrow\F_2$, we will use developments of two types:
\begin{itemize}
	\item $\dev(D_f)$, the development of the \emph{support} $D_f:=\{\mathbf{x}\in\F_2^n\colon f(\mathbf{x})=1\}$, and
	\item $\dev(G_f)$, the development of the \emph{graph} $G_f:=\{(\mathbf{x},f(\mathbf{x})):\mathbf{x} \in \F_2^n \}$.	
\end{itemize}
For the combinatorial properties of supports and graphs of bent functions as well as for their developments we refer to~\cite[Section 3]{Pott16}. We also note the following advantage of $\dev(G_{f})$ over $\dev(D_{f})$: equivalent Boolean functions $f,f'$ on $\F_2^n$ lead to isomorphic incidence structures $\dev(G_f)$ and $\dev(G_{f'})$, but at the same time $\dev(D_f)$ and $\dev(D_{f'})$ can be non-isomorphic~\cite[Example 9.3.28]{KholoshaPott2013}. For this reason we will mostly be interested in combinatorial invariants, like $p$-ranks~\cite[p. 787]{DukesWilson2007} or Smith normal forms~\cite[p. 494]{Gockenbach:2019057}, of the incidence matrix $M(\dev(G_{f}))$.

\begin{definition}
	A diagonal matrix $D$ with non-negative entries $d_1, d_2,\dots,d_n$ such that $d_1 | d_2 | \cdots | d_n$ is called the \emph{Smith normal form} of an integral matrix $A$ of order~$n$, if there exist integral matrices $U$ and $V$ with $\det(U), \det(V) = \pm 1$, such that $UAV = D$. The diagonal entries $d_i$ are called \emph{elementary divisors} of $A$.  The $p$-rank of $A$ is the rank of $A$ over the field $\F_p$.
\end{definition}
Throughout the paper we will use the following \emph{geometric invariants} of Boolean functions $f\colon \F_2^n\rightarrow\F_2$, which are defined as follows:
\begin{itemize}
	\item $\tworank(f)$ is the $\tworank$ of $M(\dev(D_f))$, for bent functions $\tworank$s have been extensively studied in~\cite{Weng20071096,Weng2008};
	\item $\Grank(f)$ is the $\tworank$ of $M(\dev(G_f))$, $\Grank$s were mostly studied in the context of inequivalence of vectorial mappings~\cite{DBLP:conf/ima/EdelP09,EdelP09};
	\item $\snf(f)$ is the Smith normal form of the incidence matrix $M(\dev(G_f))$, given by the multiset $\snf(f)=\{*d_1^{m_1},\dots, d_{k}^{m_k}*\}$, where $d_i|d_{i+1}$ and $m_i$ is the multiplicity of $d_i$.
\end{itemize}

\noindent Finally we emphasize, that $\Grank(f)$ and $\snf(f)$ are invariants under equivalence for all Boolean functions $f\colon\F_2^n\rightarrow\F_2$, while $\tworank(f)$ is invariant under equivalence only for Boolean functions $f$ with $\deg(f)\ge2$.

\subsection{The relation between geometric invariants}\label{subsection: The relation between geometric invariants}
In this subsection we show, that $\Grank$ and $\tworank$ coincide for all non-constant Boolean functions. We also show, how a small modification of the incidence matrix $M(\dev(D_f))$ can help to compute the Smith normal form of a Boolean function $f$ in a more efficient way. Finally, we partially specify elementary divisors for bent functions. 

First, we will use the following notation for incidence matrices of developments
\begin{equation*}
M_f:=M(\dev(D_f))=(f(\mathbf{x} \oplus \mathbf{y}))_{\mathbf{x},\mathbf{y}\in\F_2^n} \mbox{ and } N_f:=M(\dev(G_f)).
\end{equation*}
Note that, since $(\mathbf{x} \oplus \mathbf{y},1)\in G_f  \Leftrightarrow f(\mathbf{x} \oplus \mathbf{y})=1$ and $(\mathbf{x} \oplus \mathbf{y},0)\in G_f \Leftrightarrow \bar{f}(\mathbf{x} \oplus \mathbf{y}) =1$, we can write $N_f$ without loss of generality as the following block-matrix, where $V_i:=\{ (\mathbf{x},i) \colon \mathbf{x} \in \F_2^n\}$ for a fixed $i\in\F_2$:
\begin{equation}\label{equation: Block-Matrix of dev(Gf)}
N_f = \ \begin{blockarray}{ccc}
V_1 & V_0 \\
\begin{block}{(cc)c}
M_{f}   & M_{\bar{f}} &  V_0\\
M_{\bar{f}} & M_{f}    &  V_1\\ 
\end{block}
\end{blockarray}.
\end{equation}
Now we summarize some well-known statements about higher-order derivatives, which we will use to show the connection between geometric invariants of Boolean functions.
\begin{result}\label{result: Properties of Higher-order Derivatives}\cite{Lai1994} Let $f$ be a Boolean function on $\F_2^n$ and $\mathbf{a}_1,\dots,\mathbf{a}_k\in\F_2^n$.
	\begin{enumerate}
		\item If $\mathbf{a}_1,\dots, \mathbf{a}_k$ are linearly dependent, then $D_{\mathbf{a}_k}D_{\mathbf{a}_{k-1}}\ldots D_{\mathbf{a}_1}f=0$.
		\item Let now $\mathbf{a}_1,\dots, \mathbf{a}_k$ be linearly independent. The derivatives of $f$ are independent of the order in which the derivation is taken, i.e. the equality  $$D_{\mathbf{a}_k}D_{\mathbf{a}_{k-1}}\ldots D_{\mathbf{a}_1}f(\mathbf{x}) = D_{\mathbf{a}_{\pi(k)}}D_{\mathbf{a}_{\pi(k-1)}}\ldots D_{\mathbf{a}_{\pi(1)}}f(\mathbf{x})= \bigoplus\limits_{\mathbf{a}\in \langle \mathbf{a}_1,\dots,\mathbf{a}_k \rangle}f(\mathbf{x}\oplus \mathbf{a})$$ 
		holds for any permutation $\pi$ on $\{ 1,\ldots,k \}$.
	\end{enumerate}	
\end{result}
In the next theorem we prove that for Boolean functions of degree at least two the $\Grank$ and $\tworank$ coincide and show, that all the information about the $\snf(f)$ can be recovered from a matrix obtained through a small modification of $M_f$.
\begin{theorem}\label{theorem: Relations between two and gamma ranks}
	Let $f$ be a Boolean function on $\F_2^n$. Then the following hold:
	\begin{enumerate}
		\item If $\deg(f)\ge 1$, then the all-one-vector $\mathbf{j}_{2^{n}}$ can be expressed as a sum of an even number of vectors from the linear code $\mathcal{C}(\dev(D_f))$.
		\item If $\deg(f)<1$, then $\Grank(f)=2$, otherwise $\Grank(f)=\rank(f)$.
		\item $\snf(f)=\{*d_1^{m_1},\dots, d_{k}^{m_k},0^{2^n-1}*\}$, where all $d_i$'s are elementary divisors of the matrix 
		$\left(
		\begin{array}{cc}
		M_{f} & \mathbf{j}_{2^n}^T \\
		\mathbf{j}_{2^n} & 2 \\
		\end{array}
		\right)$.
	\end{enumerate}
\end{theorem}

\begin{proof}
	
	\emph{1.} It was shown in~\cite[Lemma 3.1]{Weng20071096}, that $\mathbf{j}_{2^n}\in\mathcal{C}(\dev(D_f))$. We will prove this statement, by expressing $\mathbf{j}_{2^n}$ as a sum of an even number of vectors from the linear code $\mathcal{C}(\dev(D_f))$. Let $d$ denotes the degree of a function $f$. First, we observe that the number of slow points of a function $f$ is bounded from below by $2^n-2^{n-d}$. Thus there exist a sequence of slow points $\mathbf{a}_1,\dots,\mathbf{a}_d$, such that the $d$-th order derivative $D_{\mathbf{a}_d}D_{\mathbf{a}_{d-1}}\ldots D_{\mathbf{a}_1}f$ is the constant one function. Finally since the following equality holds for all $\mathbf{x}\in\F_2^n$ due to Result~\ref{result: Properties of Higher-order Derivatives}
	\begin{equation*}
	D_{\mathbf{a}_d}D_{\mathbf{a}_{d-1}}\ldots D_{\mathbf{a}_1}f(\mathbf{x}) = \bigoplus\limits_{\mathbf{a} \in \langle \mathbf{a}_1,\dots, \mathbf{a}_d \rangle}f(\mathbf{x}\oplus \mathbf{a})=1,
	\end{equation*}
	one can see, the all-one-vector $\mathbf{j}_{2^n}$ is as a sum of $2^d$ elements of $\mathcal{C}(\dev(D_f))$.
	
	\noindent
	\emph{2.} Assume that the matrix $N_f$ is of the form~\eqref{equation: Block-Matrix of dev(Gf)}. Performing elementary row and column operations one can bring the matrix $N_f$ to the form
	\begin{equation*}
	N_f
	\overset{\mbox{\scalebox{0.6}{(I)}}}{\rightsquigarrow}\;
	\begin{pmatrix} 
	M_{f}   &M_{\bar{f}} \\
	\mathbf{J}_{2^n} &\mathbf{J}_{2^n}
	\end{pmatrix}
	\overset{\mbox{\scalebox{0.6}{(II)}}}{\rightsquigarrow}\;
	\begin{pmatrix} 
	M_{f}   & \mathbf{J}_{2^n} \\
	\mathbf{J}_{2^n} &\mathbf{O}_{2^n}
	\end{pmatrix}.
	\end{equation*}
	Note, that elementary column operations change the linear code $\mathcal{C}(\dev(D_f))$, however its dimension, which is equal to $\Grank(f)$, remains the same. If $\deg(f)<1$, i.e. $f$ is a constant function, clearly $\Grank(f)=2$. By the previous statement $\mathbf{j}_{2^n}$ can be expressed as a sum of an even number of rows of $M_f$. Since the matrix $M_f$ is symmetric, the vector $\mathbf{j}_{2^n}^T$ can be expressed as a sum of an even number of columns of the matrix $M_f$. In this way, the matrix $N_f$ can be brought to the form
	\begin{equation*}
	N_f
	\overset{\mbox{\scalebox{0.6}{(I)-(II)}}}{\rightsquigarrow}\;
	\begin{pmatrix} 
	M_{f}   & \mathbf{J}_{2^n} \\
	\mathbf{J}_{2^n} &\mathbf{O}_{2^n}
	\end{pmatrix}
	\overset{\mbox{\scalebox{0.6}{(III)}}}{\rightsquigarrow}\;
	\begin{pmatrix} 
	M_{f}   & \mathbf{O}_{2^n} \\
	\mathbf{O}_{2^n} &\mathbf{O}_{2^n}
	\end{pmatrix}
	\end{equation*}
	and hence $\Grank(f)=\rank(f)$. 
	
	\noindent\emph{3.} 
	Performing elementary row and column operations, as in the proof of the previous statement, but over the ring $\Z$, one can bring the matrix $N_f$ to the form
	\begin{equation*}
	N_f\rightsquigarrow
	\left(
	\begin{array}{c|c}
	\begin{array}{cc}
	M_{f} & \mathbf{j}_{2^n}^T \\
	\mathbf{j}_{2^n} & 2 \\
	\end{array} & \mathbf{O}_{2^n+1,2^n-1} \\
	\hline
	\mathbf{O}_{2^n-1,2^n+1} & \mathbf{O}_{2^n-1,2^n-1} 
	\end{array}
	\right).
	\end{equation*}
	In this way, $\snf(f)=\{*d_1^{m_1},\dots, d_{k}^{m_k},0^{2^n-1} *\}$, where $d_i$'s are elementary divisors of the matrix 
	$\left(
	\begin{array}{cc}
	M_{f} & \mathbf{j}_{2^n}^T \\
	\mathbf{j}_{2^n} & 2 \\
	\end{array}
	\right)$.	
\end{proof}

\noindent In the following proposition we partially specify the SNF of a bent function.

\begin{proposition}\label{proposition: All information about SNF}
	Let $f$ be a bent function on $\F_2^n$ and its Smith normal form given by  $\snf(f)=\{*d_1^{m_1},\dots, d_{k}^{m_k},0^{2^n-1} *\}$. Then the following holds.
	\begin{enumerate}
		\item All elementary divisors $d_i$ in the $\snf(f)$ are powers of two.
		\item $\Grank(f)=m_1$, where $m_1$ is the multiplicity of one in the $\snf(f)$.
	\end{enumerate}
\end{proposition}

\begin{proof}
	\noindent \emph{1.} Let $d_1|d_2|\ldots|d_{2^{n+1}}$ be elementary divisors and $\alpha_1,\alpha_2,\ldots,\alpha_{2^{n+1}}$ be eigenvalues of the matrix $N_f$ respectively. By~\cite[Theorem 6]{NEWMAN19911}, for all $1\le i_1<\dots<i_k\le2^{n+1}$ and $k=1,\ldots,2^{n+1}-1$ the following relation between products of elementary divisors and eigenvalues holds: $d_1\cdots d_k | \alpha_{i_1}\cdots \alpha_{i_k}$. Since $\alpha_{i_1}\cdots \alpha_{i_k} | \alpha_{i_1}^2\cdots \alpha_{i_k}^2$ it is enough to show, that all nonzero $\alpha_{i}^2$ are powers of two. Since $N_f$ is symmetric, we have $N_f^2=N_fN_f^T$. By~\cite[Lemma 1.1.4]{Pott1995FiniteGeometry}, the matrix $N_fN_f^T$ has eigenvalue $2^{2n}$ (multiplicity 1), $2^{n}$ (multiplicity $2^{n}$) and $0$ (multiplicity $2^{n}-1$). Thus the product of any $k$ nonzero elementary divisors of $N_f$ is $2^l$ for some $l$, and hence all $d_i$ are powers of two. Finally, since the $p$-rank is the number of elementary divisors, coprime with $p$ and all elementary divisors are powers of two, we conclude that $\Grank(f)=m_1$. 
\end{proof}

\begin{remark}\label{remark: Symmetry in SNF}
	We computed $\snf(f)$ for many $n$-variable bent functions of different degrees on $\F_2^n$ with $6\le n\le12$. Based on our numerical experiments, we observe the following kind of symmetry in the $\snf(f)$ of a bent function $f$ on $\F_2^n$:
	\begin{enumerate}
		\item  $\snf(f)=\{*d_1^{m_1},\dots, d_{n}^{m_n},0^{2^n-1}*\}$, where all elementary divisors $d_i$ are of the form $d_i=2^{i-1}$ for $i=1,\dots,n$.
		\item Multiplicities of elementary divisors $m_i$ satisfy $m_n=1,\;m_{n-1}=m_1-2$ and $ m_{n/2-i}=m_{n/2+i}$ for $i=1,\dots,n/2-2$.	
	\end{enumerate}
\end{remark}
We do not know how to prove this statement in general and we make the following conjecture.
\begin{conjecture}\label{conjecture: Symmetry in SNF}
	The $\snf(f)$ of a bent function $f$ on $\F_2^n$ satisfies Remark~\ref{remark: Symmetry in SNF}.
\end{conjecture}

\section{Homogeneous cubic bent functions}\label{section: Homogeneous cubic bent functions}
In this section we first survey the known homogeneous cubic bent functions. We also classify the known examples in 10 and 12 variables, constructed in~\cite{Charnes2002,MengNovel2004} by using sophisticated computational approaches, and show that:
\begin{itemize}
	\item Some of them are not covered by the Maiorana-McFarland construction;
	\item All of them are not equivalent to the only one known analytic construction (for this reason we will call it later ``the primary construction'') of Seberry, Xia and Pieprzyk, given in~\cite{SeberryConstruction}.
\end{itemize} Subsequently, we extend the latter result to an arbitrary number of variables, by proving, that proper concatenations of homogeneous cubic bent functions in a small number of variables can never be equivalent to the primary construction. Finally we provide a construction method, aimed to generate a lot of homogeneous bent functions from a single given example. Using this approach we construct many new homogeneous cubic bent functions in 12 variables and show, that some of them are not equivalent to all the previously known ones.
\subsection{The known examples and constructions}\label{subsection: The known examples and constructions}
The existence of homogeneous cubic bent functions on $\F_2^n$ for all $n\ge6$ was shown in two independent ways. Seberry, Xia and Pieprzyk in~\cite[Theorem 8]{SeberryConstruction} proved that one can construct such functions on $\F_2^n$ for all even $n\ne 8$, from special Maiorana-McFarland functions by a proper change of basis. We will call their construction \emph{primary} and denote any $n$-variable function of this type by $h^n_{pr.}$.

\begin{result}\label{result: Seberry et al. construction}
	~\cite[Theorem 6]{SeberryConstruction} Let $f_{id,\phi}$ be a Maiorana-McFarland bent function on $\F_2^{2m}$	where $\phi$ is a homogeneous cubic function without affine derivatives on $\F_2^{m}$. Then there exists a nonsingular matrix $T$, such that $h^n_{pr.}(\mathbf{x},\mathbf{y}):=f_{id,\phi}((\mathbf{x},\mathbf{y})T)$ is a homogeneous cubic bent function.
\end{result}

Another approach, suggested by Charnes, R\"otteler and Beth in~\cite{Charnes2002}, consists of two steps. First, they constructed homogeneous cubic bent functions in a small number of variables using the tools from modular invariant theory, and second, they extended these examples to an arbitrary number of variables, using the direct sum construction.

\begin{result}\label{result: Direct sum of hom. bents}\cite[Theorem 2]{SeberryConstruction}
	The direct sum $h(\mathbf{x},\mathbf{y})=f(\mathbf{x})\oplus g(\mathbf{y})$ is $d$-homogeneous bent on $\F_2^{n+m}$ if and only if the functions $f$ and $g$ are $d$-homogeneous bent on $\F_2^n$ and $\F_2^m$ respectively.
\end{result}

\noindent Further we classify the known homogeneous cubic bent functions in a small number of variables and show, that some of them are not the members of the $\MClass$ class.
\begin{theorem}\label{theorem: Analysis in 10,12 variables}
	The homogeneous cubic bent functions in $n=10$ or $n=12$ variables from~\cite[p. 149]{Charnes2002} and~\cite[p. 15]{MengNovel2004} satisfy:
	\begin{enumerate}
		\item If $n=10$, there are 4 equivalence classes, with 2 of them being outside the completed Maiorana-McFarland class $\MClass$.
		\item If $n=12$, there are 5 equivalence classes, which are subclasses of $\MClass$.
	\end{enumerate}
\end{theorem}
\begin{proof}
	First, we compute the Smith normal forms for the mentioned homogeneous cubic bent functions and check whether those, having the same ones, are equivalent. We check equivalence of bent functions via equivalence of linear codes~\cite[Theorem 9]{EdelP09} and isomorphism of designs~\cite[Corollary 10.6]{Bending1993} in \texttt{Magma}~\cite{MR1484478}. Consequently, we found 4 and 5 equivalence classes in 10 and 12 variables, respectively. We denote representatives of the obtained classes by  $h^n_i$ and list them in the Appendix~\ref{section: Appendix}.  We provide only the first $n/2$ elementary divisors for the Smith normal forms of bent functions due to Remark~\ref{remark: Symmetry in SNF}.
	\begin{table}[H]
		\caption{First $n/2$ elementary divisors of the Smith normal form $\snf(h^n_i)$ for the known homogeneous cubic bent functions from~\cite[p. 149]{Charnes2002} and~\cite[p. 15]{MengNovel2004}.}
		\label{table: Invariant SNF}
		\centering
		\begin{subtable}[t]{.45\linewidth}
			\centering
			\scalebox{0.93}{
				\begin{tabular}{|c|l|}
					\hline
					$h^{10}_i$ & \multicolumn{1}{c|}{$\snf(h^{10}_i)$}  \\ \hline			
					$h^{10}_1$ & $\{*1^{20}, 2^{86}, 4^{130}, 8^{143}, 16^{268},\dots *\}$ \\ \hline
					$h^{10}_2$ & $\{*1^{20}, 2^{78}, 4^{138}, 8^{147}, 16^{260},\dots *\}$  \\ \hline
					$h^{10}_3$ & $\{*1^{20}, 2^{108}, 4^{110}, 8^{129}, 16^{292},\dots *\}$ \\ \hline
					$h^{10}_4$ & $\{*1^{22}, 2^{154}, 4^{90}, 8^{81}, 16^{332},\dots *\}$   \\ \hline
			\end{tabular}}
		\end{subtable}%
		\begin{subtable}[t]{.55\linewidth}
			\centering
			\scalebox{0.93}{
				\begin{tabular}{|c|l|}
					\hline
					$h^{12}_i$    & \multicolumn{1}{c|}{$\snf(h^{12}_i)$}  \\ \hline
					$h^{12}_{1}$ & $\{*1^{22}, 2^{142}, 4^{276}, 8^{493}, 16^{630}, 32^{972},\dots *\}$  \\ \hline
					$h^{12}_{2}$ & $\{*1^{22}, 2^{126}, 4^{276}, 8^{517}, 16^{646}, 32^{924},\dots *\}$  \\ \hline
					$h^{12}_3$ & $\{*1^{24}, 2^{127}, 4^{260}, 8^{525},16^{674}, 32^{878},\dots *\}$  \\ \hline
					$h^{12}_4$ & $\{*1^{22}, 2^{104}, 4^{256}, 8^{525}, 16^{698}, 32^{888},\dots *\}$  \\ \hline
					$h^{12}_5$ & $\{*1^{26}, 2^{196}, 4^{392}, 8^{419}, 16^{490}, 32^{1052},\dots *\}$ \\ \hline		
			\end{tabular}}
		\end{subtable}%
	\end{table}
	\noindent Further we use the parallel implementation of Algorithm~\ref{algorithm: f in M} in \texttt{Mathematica}~\cite{WolframMathematica112} in order to check, whether the functions $h^n_i$ belong to $\MClass$. As a result, only functions $h^{10}_3$ and $h^{10}_4$ do not belong to the $\MClass$ class, while all the functions $h^{12}_i$ are in $\MClass$. Finally, we list all the $\mathcal{M}$-subspaces of functions from $\MClass$ in the Appendix~\ref{section: Appendix}.
\end{proof}

\subsection{Homogeneous cubic bent functions, different from the primary construction}\label{subsection: Homogeneous cubic bent functions, different from the primary construction}
Using the facts about $\tworank$s and the relation between $\Grank$ and $\tworank$, obtained in the previous section, we derive the following corollary.
\begin{corollary}\label{corollary: Granks of bent functions}
	Let $f$ and $g$ be Boolean functions on $\F_2^n$ and $\F_2^m$, respectively, with $\deg(f)\ge 1$ and $\deg(g) \ge 1$. 
	\begin{enumerate}
		\item  Let $h$ be a Boolean function on $\F_2^{n}\times\F_2^{m}$ defined as the direct sum of functions $f$ and $g$, then
		\begin{equation}\label{equation: Gamma-rank of the direct sum}
		\Grank(h)=\Grank(f)+\Grank(g)-2.
		\end{equation}
		\item Let $f_{id,\phi}$ be a Maiorana-McFarland bent function on $\F_2^n$, then
		\begin{equation}\label{equation: Minimal Rank Functions}
		\Grank(f_{id,\phi})=n+2 \mbox{ if and only if }\deg(\phi)\le3.
		\end{equation}
		\item For the primary construction of homogeneous cubic bent functions $h^n_{pr.}$ on $\F_2^n$ we have \\
		$\Grank(h^n_{pr.})= n+2$.
	\end{enumerate}
\end{corollary}
\begin{proof}
	The first and the second claims hold, since the statements~\eqref{equation: Gamma-rank of the direct sum} and~\eqref{equation: Minimal Rank Functions} were proven in~\cite{Weng20071096,Weng2008} for $\tworank$s, and by Theorem~\ref{theorem: Relations between two and gamma ranks} we know, that $\tworank$s and $\Grank$s coincide for all non-constant Boolean functions. Finally, the third claim follows from~\eqref{equation: Minimal Rank Functions} and the definition of the primary construction.
\end{proof}
\noindent  Now we proof the existence of homogeneous cubic bent functions, different from the primary construction. 
\begin{theorem}\label{theorem: Non-Seberry functions}
	There exist homogeneous cubic bent functions on $\F_2^n$, inequivalent to the primary construction $h^n_{pr.}$, whenever $n\ge8$. 
\end{theorem}
\begin{proof}
	We construct a homogeneous cubic bent function $h_n$ in $n=6i+8j+10k+12l$ variables with $j+k+l\neq0$ as the following concatenation:
	\begin{equation}
	h_n:=i\cdot h^6_* \oplus j\cdot h^8_* \oplus k\cdot h^{10}_* \oplus l\cdot h^{12}_*,
	\end{equation}
	where $h^6_*$ and $h^8_*$ are arbitrary homogeneous cubic bent functions in $6$ and $8$ variables respectively, and $h^{10}_*,h^{12}_*$ are arbitrary homogeneous cubic bent functions in $10$ and $12$ variables from Table~\ref{table: Invariant SNF}. Since any homogeneous cubic bent function in $6$ variables is equivalent to the primary construction $h^6_{pr.}$, we have $\Grank(h^6_*)=8$. One can check that for any cubic bent function $h^8_*$ in $8$ variables we have $\Grank(h^8_*)\in\{14,16\}$. By Proposition~\ref{proposition: All information about SNF} one can see, that $\Grank$s of functions $h^{10}_*$ and $h^{12}_*$ are multiplicities of the entry one in Table~\ref{table: Invariant SNF}. Finally, comparing the lower bound of the $\Grank(h_n)$ with $\Grank(h^n_{pr.})$, one can see immediately that
	$$
	\begin{aligned}
	\Grank(h_n) & \ge 8i + 14j + 20 k + 22 l - 2 (i + j + k + l-1) \\
	& = n + 2 + 4(j+2(k+l))>n+2=\Grank(h^n_{pr.})
	\end{aligned}$$
	and hence the function $h_n$ is never equivalent to $h^n_{pr.}$ for all $n\ge 8$.
\end{proof}

\subsection{Constructing new homogeneous functions from old, without increasing the number of variables}\label{subsection: New from old}
In this subsection we show, that in some cases one can use the power of the Maiorana-McFarland construction to produce a lot of homogeneous bent functions, provided that a single one, member of the $\MClass$ class, is given. Our approach is based on a generalization of the following observation.
\begin{observation}
	Let $f:=h^{12}_3$ and $g:=h^{12}_4$. Our computations show, that homogeneous cubic bent functions $f$ and $g$ have a common $\mathcal{M}$-subspace $U$ of dimension $6$, which together with its complement $\bar{U}$ is given by:
	\begin{equation}\label{equation: MSubspace from the Observation}
	\gjb(U)=\left(\begin{array}{c|c|c}
	1 \; 1     & \multicolumn{2}{c}{\mathbf{O}_{1,10}} \\ \hline
	\mathbf{O}_{5,2} & \mathbf{I}_5      & \mathbf{I}_5 
	\end{array}\right)\quad\mbox{and}\quad
	\gjb(\bar{U})=\left(\begin{array}{c|c|c}
	0 \; 1     & \multicolumn{2}{c}{\mathbf{O}_{1,10}} \\ \hline
	\mathbf{O}_{5,2} & \mathbf{O}_5      & \mathbf{I}_5 
	\end{array}\right).
	\end{equation}
	By Remark~\ref{remark: How to construct a linear mapping} one can bring functions $f$ and $g$ to their Maiorana-McFarland representations~\eqref{equation: Maiorana-McFarland Representation} using the same linear invertible transformation $A_U$, given by~\eqref{equation: Linear Transform}:
	$$f(\mathbf{z}A_U)=f_{\pi,\phi}(\mathbf{x},\mathbf{y})\quad\mbox{and}\quad g(\mathbf{z}A_U)=g_{\pi,\psi}(\mathbf{x},\mathbf{y}),$$
	where $\pi\colon\F_2^6\rightarrow\F_2^6$ is a permutation and $\phi,\psi\colon\F_2^6\rightarrow\F_2$ are Boolean functions. In this way, one can construct homogeneous function $g$ from the function $f$ as follows:
	\begin{equation}\label{equation: Very Good Linear Transformation}
	g(\mathbf{z}):=f_{\pi,\phi\oplus\omega}((\mathbf{x},\mathbf{y})T),\mbox{ where }\omega:=\phi\oplus\psi \mbox{ and }T:=A^{-1}_U.
	\end{equation}
	
\end{observation}
Let $h_{\pi,\phi}\colon\F_2^n\rightarrow\F_2$ be a bent function from the $\MClass_{r,s}$ class, which is equivalent to a  $d$-homogeneous one, i.e. there exist an invertible matrix $T$ of order $n$, such that $h_{\pi,\phi}((\mathbf{x},\mathbf{y})T)$ is $d$-homogeneous. We will denote by $\Omega_T(h_{\pi,\phi})$ the set
\begin{equation*}
\Omega_T(h_{\pi,\phi}):=\{ \omega\colon\F_2^s\rightarrow\F_2 \; | \; h_{\pi,\phi\oplus\omega}((\mathbf{x},\mathbf{y})T) \mbox{ is }d\mbox{-homogeneous bent}\}.
\end{equation*}
This is the set of all Boolean functions $\omega$ on $\F_2^s$, which preserve $d$-homogeneity and bentness of the function $h_{\pi,\phi\oplus\omega}$ with respect to the linear transformation $T$. 
\begin{proposition}\label{proposition: Ideal of Functions}
	Let $h_{\pi,\phi}$ be a Maiorana-McFarland bent function on $\F_2^{2m}$, which is equivalent to a $d$-homogeneous bent function, i.e. there exist an invertible matrix $T$, such that $h_{\pi,\phi}((\mathbf{x},\mathbf{y})T)$ is $d$-homogeneous bent. Then the set $\Omega_T(h_{\pi,\phi})$ is a vector space over $\F_2$.
\end{proposition}
\begin{proof}
	Let $\omega_1,\omega_2\in \Omega_T(h_{\pi,\phi})$ with $\omega_1\neq\omega_2$ and $\omega:=\omega_1\oplus\omega_2$. We will show that $\omega\in \Omega_T(h_{\pi,\phi})$. Let the invertible matrix $T$ be of the form $T=\left(\begin{array}{cc}A & B\\C & D\end{array}\right)$ with all the submatrices of order $m$. First, we observe that $0\in\Omega_T(h_{\pi,\phi})$ and for any $\omega_i\in  \Omega_T(h_{\pi,\phi})$ we have 
	$$ h_{\pi_,\phi\oplus\omega_i}((\mathbf{x},\mathbf{y})T)=h_{\pi_,\phi}((\mathbf{x},\mathbf{y})T)\oplus \omega_i(\mathbf{x}B \oplus \mathbf{y}D),$$
	from what follows, that $\omega_i(\mathbf{x}B \oplus \mathbf{y}D)$ is either $d$-homogeneous or constant zero function, since $h_{\pi_,\phi}((\mathbf{x},\mathbf{y})T)$ is $d$-homogeneous. Thus $\omega\in \Omega_T(h_{\pi,\phi})$, since bentness of $h_{\pi_,\phi\oplus\omega}$ is independent on the choice of a function $\omega$ on $\F_2^m$ and $\omega(\mathbf{x}B \oplus \mathbf{y}D)$ is a $d$-homogeneous function.
\end{proof}

Note that for a homogeneous bent function $h_{\pi,\phi}\in\MClass_{r,s}$ the set $\Omega_T(h_{\pi,\phi})$ is not a vector space in general. Nevertheless, for a given homogeneous bent function $h\in\MClass_{r,s}$ one can still construct the set $\Omega_T(h_{\pi,\phi})$, in order to get more, possibly inequivalent, homogeneous functions. We will summarize these ideas in the form of an algorithm below.
\begin{algorithm}[H]
	\caption{New $d$-homogeneous bent functions from a single one in $\MrsClass$.}
	\label{algorithm: New from Old}
	\begin{algorithmic}[1]
		\Require Homogeneous bent function $h:\F_2^{n}\rightarrow\F_2,h\in\MrsClass$ of degree $d$.
		\Ensure The set $H$ of new $d$-homogeneous bent functions from $\MrsClass$.
		\State \textbf{Put} $H\gets \{ \}$.
		\ForAll{$\mathcal{M}$-subspaces $U \in \mathcal{MS}_r(h)$}		  
		\State \textbf{Construct} a linear mapping $A_U$ as in Remark~\ref{remark: How to construct a linear mapping}, in order to get the \Indent Maiorana-McFarland representation~\eqref{equation: Maiorana-McFarland Representation}, i.e. $h_{\pi,\phi}(\mathbf{x},\mathbf{y}):=h(\mathbf{z}A_U)$. \EndIndent
		\State \textbf{Put} $H\gets H\cup \{ h_{\pi,\phi\oplus\omega}((\mathbf{x},\mathbf{y})T)\colon \omega \in \Omega_{T}(h_{\pi,\phi}) \} $, where $T:=A_{U}^{-1}$.
		\EndFor
	\end{algorithmic}
\end{algorithm}
\begin{remark}\label{remark: New homogeneous functions}
	Using Algorithm~\ref{algorithm: New from Old} and the mapping $T$, defined in~\eqref{equation: Very Good Linear Transformation}, one can construct $2^{\binom{6}{3}}$ new homogeneous cubic bent functions from any of functions $h^{12}_3$ and $h^{12}_5$, members of the $\MClass$ class. Such a big number of new functions can be explained in the following way. Let $h\in \{h^{12}_3, h^{12}_5\}$. First, we observe that the image of $\mathbf{y}$ after the linear transformation $\mathbf{y}\mapsto\mathbf{y}'=\mathbf{x}B \oplus \mathbf{y}D$ is given by:
	\begin{equation}\label{equation: Image of y}
	\mathbf{y}\mapsto\mathbf{y}'=(x_1 \oplus x_2, x_3 \oplus y_2, x_4 \oplus y_3, x_5 \oplus y_4, x_6 \oplus y_5, y_1 \oplus y_6).
	\end{equation}
	Since any two coordinates of the vector $\mathbf{y}'$ do not contain common variables $x_i$ and $y_j$, the linear transformation, defined in~\eqref{equation: Image of y}, is homogeneity-preserving. Thus, $\Omega_{T}(h_{\pi,\phi})$ is generated by monomials $\omega\colon\F_2^6\rightarrow\F_2$ of degree $3$, and hence $|\Omega_T(h_{\pi,\phi})|=2^{\binom{6}{3}}$. Finally, we note that some of the constructed homogeneous cubic bent functions are not equivalent to any of the known one, since their Smith normal forms, listed in Table~\ref{table: New Homogeneous Cubic Bent Functions}, are different from those given in Table~\ref{table: Invariant SNF}.
	\begin{table}[H]
		\caption{First $n/2$ elementary divisors of the Smith normal form $\snf(h^n_i)$ for the new homogeneous cubic bent functions $h^{12}_{6},h^{12}_{7}$ in $12$ variables.}
		\label{table: New Homogeneous Cubic Bent Functions}
		\centering
		\scalebox{1}{
			\begin{tabular}{|c|l|}
				\hline
				$h^{12}_i$    & \multicolumn{1}{c|}{$\snf(h^{12}_i)$}  \\ \hline
				$h^{12}_{6}$ & $\{*1^{24}, 2^{123}, 4^{292}, 8^{497}, 16^{674}, 32^{878},\dots *\}$  \\ \hline
				$h^{12}_{7}$ & $\{*1^{24}, 2^{123}, 4^{272}, 8^{516}, 16^{674}, 32^{880},\dots *\}$  \\ \hline
			\end{tabular}
		} 
	\end{table}
\end{remark}

\begin{theorem}
	There are at least 7 pairwise inequivalent homogeneous cubic bent functions on $\F_2^{12}$, inequivalent to $h^{12}_{pr.}$.
\end{theorem}
Finally we want to emphasize the fundamental difference between the primary construction $h^n_{pr.}$ and functions, constructed in Remark~\ref{remark: New homogeneous functions}. For the primary construction of homogeneous cubic bent function $h^n_{pr.}$ one needs to find a special Boolean function $\phi$ of degree 3, such that the non-homogeneous cubic Maiorana-McFarland function $f_{id,\phi}$ is homogeneous after the change of coordinates. In some sense, the identity permutation $id$ has a ``defect'', which makes $f_{id,0}$ never equivalent to a homogeneous cubic function.  But the specific choice of a cubic function $\phi$ helps to repair it. Since the functions constructed in Remark~\ref{remark: New homogeneous functions} are in that sense ``defect free'', it is essential to construct such functions systematically.
\begin{openproblem}
	Are there infinite families of permutations $\pi\colon\F_2^m\rightarrow\F_2^m$, such that for some non-degenerate linear transformation $T$ the function $f_{\pi,\psi}((\mathbf{x},\mathbf{y})T)$ is homogeneous cubic bent for all homogeneous cubic functions $\psi\colon\F_2^m\rightarrow\F_2$?
\end{openproblem}
\section{Bent functions outside the $\MClass$ class via direct sum construction}\label{section: Bent functions outside M}
In this section we show how one can choose bent functions $f$ and $g$, such that the direct sum $f\oplus g$ is not a member of the completed Maiorana-McFarland class $\MClass$. The idea of the approach is based on the following observation: if one can measure the maximum dimension of \emph{relaxed} $\mathcal{M}$\emph{-subspaces} (which we introduce below) of the components $f$ and $g$, then one can provide an upper bound for the linearity index $\ind(f\oplus g)$ and if it small enough, then $f\oplus g\notin \MClass$. 

Finally, using this recursive approach, we prove the series of results about the existence of cubic bent functions outside the $\MClass$ class, which can simultaneously be homogeneous and have no affine derivatives.
\subsection{The sufficient condition in terms of relaxed $\mathcal{M}$-subspaces}\label{subsection: MSubspaces}
Further, we identify $\F_2^{n+m}$ with $\F_2^n\times \F_2^m$. In this way, any vector $\mathbf{v}\in\F_2^{n+m}$ is uniquely represented by a pair $(\mathbf{v}_\mathbf{x},\mathbf{v}_\mathbf{y})$, where $\mathbf{v}_\mathbf{x}\in\F_2^n$ and $\mathbf{v}_\mathbf{y}\in\F_2^m$. Now let $U\in\mathcal{MS}(h)$, i.e. for all $\mathbf{a},\mathbf{b}\in U$ we have, that second-order derivatives satisfy $D_{\mathbf{a},\mathbf{b}}h=0$. This takes place if and only if $D_{\mathbf{a}_{\mathbf{x}},\mathbf{b}_{\mathbf{x}}}f=D_{\mathbf{a}_\mathbf{y},\mathbf{b}_\mathbf{y}}g=c_{\mathbf{a},\mathbf{b}}$, where $c_{\mathbf{a},\mathbf{b}}\in\F_2$ is a constant, depending on $\mathbf{a}$ and $\mathbf{b}$, since $g$ and $h$ do not have common variables. This observation leads to the following generalization of $\mathcal{M}$-subspaces (see Definition~\ref{definition: MSubspace}).
\begin{definition}  We will call a vector subspace $U$ a \emph{relaxed} $\mathcal{M}$\emph{-subspace} of a Boolean function $f\colon\F_2^n\rightarrow\F_2$, if for all $\mathbf{a},\mathbf{b}\in U$ second order derivatives $D_{\mathbf{a},\mathbf{b}}f$ are either constant zero or constant one functions, i.e $D_{\mathbf{a},\mathbf{b}}f=0$ or $D_{\mathbf{a},\mathbf{b}}f=1$.  We denote by $\mathcal{RMS}_r(f)$ the collection of all $r$-dimensional relaxed $\mathcal{M}$-subspaces of $f$ and by $\mathcal{RMS}(f)$ the collection $$\mathcal{RMS}(f):=\bigcup\limits_{r=1}^{n} \mathcal{RMS}_r(f).$$
\end{definition}
While the linearity index of a Boolean function (see Definition~\ref{definition: Linearity Index}) is defined as the maximal possible dimension of its $\mathcal{M}$-subspace, it is reasonable to define its analogue for relaxed $\mathcal{M}$-subspaces.
\begin{definition}
	For a Boolean function $f\colon\F_2^n\rightarrow\F_2$ its \emph{relaxed linearity index} $\rind(f)$ is defined by  $\rind(f):=\max\limits_{U\in \mathcal{RMS}(f)}\dim(U)$.
\end{definition}

\begin{example}
	Let $f\colon\F_2^6\rightarrow\F_2$ be the function from Example~\ref{example: M-subspace}. One can check, that the subspace $U=\scalebox{1}{$\langle (0, 1, 0, 0, 0, 1), (0, 0, 0, 1, 0, 0), (0, 0, 0, 0, 1, 1)\rangle$}$ is a relaxed $\mathcal{M}$-subspace of $f$, since its second-order derivatives $D_{\mathbf{a},\mathbf{b}}f$, which correspond to all two-dimensional vector subspaces $\langle\mathbf{a},\mathbf{b}\rangle $ of $U$, are constant zero or constant one functions
	$$\scalebox{0.83}{$
	\begin{gathered}
	\scalebox{1}{$\left\langle
		\begin{array}{cccccc}
		0 & 0 & 0 & 1 & 0 & 0 \\
		0 & 0 & 0 & 0 & 1 & 1 \\
		\end{array}
		\right\rangle\mapsto 0$},
	\scalebox{1}{$\left\langle
		\begin{array}{cccccc}
		0 & 1 & 0 & 0 & 0 & 1 \\
		0 & 0 & 0 & 0 & 1 & 1 \\
		\end{array}
		\right\rangle\mapsto 1$},
	\scalebox{1}{$\left\langle
		\begin{array}{cccccc}
		0 & 1 & 0 & 1 & 0 & 1 \\
		0 & 0 & 0 & 0 & 1 & 1 \\
		\end{array}
		\right\rangle\mapsto 1$},
	\scalebox{1}{$\left\langle
		\begin{array}{cccccc}
		0 & 1 & 0 & 0 & 0 & 1 \\
		0 & 0 & 0 & 1 & 0 & 0 \\
		\end{array}
		\right\rangle\mapsto 0$},
	\\
	\scalebox{1}{$\left\langle
		\begin{array}{cccccc}
		0 & 1 & 0 & 0 & 1 & 0 \\
		0 & 0 & 0 & 1 & 0 & 0 \\
		\end{array}
		\right\rangle\mapsto 0$},
	\scalebox{1}{$\left\langle
		\begin{array}{cccccc}
		0 & 1 & 0 & 0 & 0 & 1 \\
		0 & 0 & 0 & 1 & 1 & 1 \\
		\end{array}
		\right\rangle\mapsto 1$},
	\scalebox{1}{$\left\langle
		\begin{array}{cccccc}
		0 & 1 & 0 & 0 & 1 & 0 \\
		0 & 0 & 0 & 1 & 1 & 1 \\
		\end{array}
		\right\rangle\mapsto 1$}.
	\end{gathered}$}
	$$
\end{example}

Now we present some properties of collections of $\mathcal{M}$-subspaces as well as of relaxed ones.
\begin{proposition}
	Let $f\colon\F_2^n\rightarrow\F_2$ be a Boolean function and let $n=r+s$. The following hold:
	\begin{enumerate}
		\item $\mathcal{MS}(f)\subseteq \mathcal{RMS}(f)$.
		\item $\left| \mathcal{MS}_r(f) \right|$ and $\left| \mathcal{RMS}_r(f) \right|$ as well as $\ind(f)$ and $ \rind(f)$ are invariants under equivalence.
		\item $\ind(f)\le\rind(f)$ and $f\notin\MClass_{r,s}$ for all $r>\rind(f)$.
	\end{enumerate}		
\end{proposition}
\begin{proof}
	\emph{1.} This follows from the definitions of collections $\mathcal{MS}(f)$ and $\mathcal{RMS}(f)$.
	
	\noindent \emph{2.} Let $f$ and $f'$ be equivalent, i.e. $f'(\mathbf{x})=f(\mathbf{x}A)\oplus l(\mathbf{x})$. Assume $U\in\mathcal{RMS}_r(f)$ and 
	let $U'=UA^{-1}$ with $\mathbf{a}',\mathbf{b}'\in U'$. Denoting $\mathbf{y}=\mathbf{x}A$, one can see from the following computations
	\[\arraycolsep=1.0pt\def\arraystretch{1.0}
	\begin{array}{rcl}
	D_{\mathbf{a}',\mathbf{b}'}f'(\mathbf{x}) & = & f'(\mathbf{x}\oplus \mathbf{a}' \oplus \mathbf{b}') \oplus f'(\mathbf{x}\oplus \mathbf{a}') \oplus f'(\mathbf{x} \oplus \mathbf{b}') \oplus f'(\mathbf{x}')\\
	& = & f(\mathbf{y}\oplus \mathbf{a}\oplus \mathbf{b}) \oplus f(\mathbf{y}\oplus \mathbf{a}) \oplus f(\mathbf{y}\oplus \mathbf{b}) \oplus f(\mathbf{y})=D_{\mathbf{a},\mathbf{b}}f(\mathbf{y})\\
	\end{array}
	\]
	that $U'\in\mathcal{RMS}_r(f')$. Since $A^{-1}$ maps different subspaces to different ones, we have that $\left| \mathcal{RMS}_r(f) \right|=\left| \mathcal{RMS}_r(f') \right|$ and $\left| \mathcal{MS}_r(f) \right|=\left| \mathcal{MS}_r(f') \right|$. Since $\dim(U)=\dim(U')$, we have 
	$\ind(f)=\ind(f')$ and $ \rind(f)=\rind(f')$.
	
	\noindent\emph{3.} First, since $\mathcal{MS}(f)\subseteq \mathcal{RMS}(f)$ the inequality  $\ind(f)\le\rind(f)$ holds. The statement $f\notin\MClass_{r,s}$ for all $r>\rind(f)$ now follows from the maximality of the linearity index. 
\end{proof}

In the next theorem we will show, that each relaxed $\mathcal{M}$-subspace of $f\oplus g$ is contained in another relaxed $\mathcal{M}$-subspace from $\mathcal{RMS}(f\oplus g)$, constructed via the direct product of relaxed $\mathcal{M}$-subspaces of $f$ and $g$.

\begin{theorem}\label{theorem: MM Subspaces of Direct Sums}
	Let $h(\mathbf{x},\mathbf{y}):=f(\mathbf{x})\oplus g(\mathbf{y}),\mbox{ for }\mathbf{x}\in\F_2^{n}$ and $\mathbf{y}\in\F_2^{m}$.
	\begin{enumerate}
		\item\label{theorem: Properties of MSubspaces, part i} If $V\in \mathcal{RMS}(f)$ and $W\in \mathcal{RMS}(g)$, then $V\times W\in\mathcal{RMS}(h)$.
		\item\label{theorem: Properties of MSubspaces, part ii} For any $ U\in \mathcal{RMS}(h)$ there exist $V\in \mathcal{RMS}(f)$ and $W\in \mathcal{RMS}(g)$, such that $U\subseteq V\times W$.
		\item\label{theorem: Properties of MSubspaces, part iii} $\rind(h) \le \rind(f) + \rind(g)$.	
	\end{enumerate}	
\end{theorem}

\begin{proof}
	\emph{1.} Let $ U=V\times W$. Since $V\in\mathcal{RMS}(f)$ and $ W\in\mathcal{RMS}(g)$, then for all $\mathbf{v}_1,\mathbf{v}_2\in V$ holds $D_{\mathbf{v}_1,\mathbf{v}_2}f=c_{\mathbf{v}_1,\mathbf{v}_2}$ and for all $\mathbf{w}_1,\mathbf{w}_2\in W$ holds $D_{\mathbf{w}_1,\mathbf{w}_2}g=c_{\mathbf{w}_1,\mathbf{w}_2}$, where $c_{\mathbf{v}_1,\mathbf{v}_2}$ and $c_{\mathbf{w}_1,\mathbf{w}_2}$ are some constants. In this way, for all pairs $\mathbf{u}_1=(\mathbf{v}_1,\mathbf{w}_1)$ and $\mathbf{u}_2=(\mathbf{v}_2,\mathbf{w}_2)$ holds
	$D_{\mathbf{u}_1,\mathbf{u}_2}h=D_{\mathbf{v}_1,\mathbf{v}_2}f\oplus D_{\mathbf{w}_1,\mathbf{w}_2}g=c_{\mathbf{v}_1,\mathbf{v}_2}\oplus c_{\mathbf{w}_1,\mathbf{w}_2}$ and, hence, $ U\in\mathcal{RMS}(h)$.
	
	\noindent \emph{2.} Recall that any vector $\mathbf{v}\in\F_2^{n+m}$ is identified with a pair $(\mathbf{v}_\mathbf{x},\mathbf{v}_\mathbf{y})$, where $\mathbf{v}_\mathbf{x}\in\F_2^n$ and $\mathbf{v}_\mathbf{y}\in\F_2^m$. We define two vector subspaces $V\subseteq\F_2^n$ and $W\subseteq\F_2^m$ as follows:
	\begin{equation*}\label{equation: V and W}
	V=\spa(\{\mathbf{u}_\textbf{x}\colon \mathbf{u}\in U \})\mbox{ and } W=\spa(\{\mathbf{u}_y\colon \mathbf{u}\in U \}).
	\end{equation*}
	We will show, that $V\in\mathcal{RMS}(f)$ and $W\in\mathcal{RMS}(g)$. We define two functions $f',g'\colon\F_2^{n+m}\rightarrow\F_2^{n+m}$ as $f'(\mathbf{x},\mathbf{y}):=f(\mathbf{x})$ for all $\mathbf{y}\in\F_2^m$ and $g'(\mathbf{x},\mathbf{y}):=g(\mathbf{y})$ for all $\mathbf{x}\in\F_2^n$. Since $ U\in \mathcal{RMS}(h)$, then for all $\mathbf{u}_1,\mathbf{u}_2\in U$ the equality 
	\begin{equation}\label{equation: MM subspace of direct sum}
	D_{\mathbf{u}_1,\mathbf{u}_2}h(\mathbf{x},\mathbf{y})=D_{\mathbf{u}_1,\mathbf{u}_2}f'(\mathbf{x},\mathbf{y})\oplus D_{\mathbf{u}_1,\mathbf{u}_2}g'(\mathbf{x},\mathbf{y})=c_{\mathbf{u}_1,\mathbf{u}_2}
	\end{equation}
	holds for all $(\mathbf{x},\mathbf{y})\in\F_2^{n+m}$.
	Let $\mathbf{x}_1,\mathbf{x}_2\in\F_2^n$ and consider the following equalities
	\begin{align}
	\label{equation: first} D_{\mathbf{u}_1,\mathbf{u}_2}f'(\mathbf{x}_1,\mathbf{y})\oplus D_{\mathbf{u}_1,\mathbf{u}_2}g'(\mathbf{x}_1,\mathbf{y})= &c_{\mathbf{u}_1,\mathbf{u}_2}\\
	\label{equation: second}
	D_{\mathbf{u}_1,\mathbf{u}_2}f'(\mathbf{x}_2,\mathbf{y})\oplus D_{\mathbf{u}_1,\mathbf{u}_2}g'(\mathbf{x}_2,\mathbf{y})= &c_{\mathbf{u}_1,\mathbf{u}_2},
	\end{align}
	which hold for any $\mathbf{y}\in\F_2^m$ due to~\eqref{equation: MM subspace of direct sum}. Adding equation~\eqref{equation: first} to~\eqref{equation: second}, one gets $D_{\mathbf{u}_1,\mathbf{u}_2}f'(\mathbf{x}_1,\mathbf{y})=D_{\mathbf{u}_1,\mathbf{u}_2}f'(\mathbf{x}_2,\mathbf{y})$ since $g'$ depends on the variable $\mathbf{x}$ ``fictively''. Now, since 
	$f'$ depends on the variable $\mathbf{y}$ ``fictively'', we get that for all $\mathbf{v}_1,\mathbf{v}_2\in V$ the equality $D_{\mathbf{v}_1,\mathbf{v}_2}f(\mathbf{x}_1)=D_{\mathbf{v}_1,\mathbf{v}_2}f(\mathbf{x}_2)$ holds for all $\mathbf{x}_1,\mathbf{x}_2\in\F_2^n$ and hence $D_{\mathbf{v}_1,\mathbf{v}_2}f=c_{\mathbf{v}_1,\mathbf{v}_2}$ (one can think about $\mathbf{v}_1$ and $\mathbf{v}_2$ as $\left( \mathbf{u}_1 \right)_{\mathbf{x}}$ and $\left( \mathbf{u}_2 \right)_{\mathbf{x}}$, respectively). Thus we have shown, that $V\in\mathcal{RMS}(f)$. Since $f$ and $g$ are interchangeable, we get $W\in\mathcal{RMS}(g)$.
	Clearly, $ U\subseteq V\times W$ and by the previous statement we have $V\times W\in\mathcal{RMS}(h)$.
	
	\noindent\emph{3.} Let $U\in\mathcal{RMS}(h)$ and $\dim(U)=\rind(h)$. By the previous statement there exist $V\in \mathcal{RMS}(f)$ and $W\in \mathcal{RMS}(g)$, such that $U\subseteq V\times W$. Now, using the following series of inequalities	
	\begin{align*}
	\rind(h)=\dim(U)&\le \dim (V\times W)= \dim (V)+\dim(W)\\
	&\le \max\limits_{V\in \mathcal{RMS}(f)}\dim\left(V\right)+ \max\limits_{W\in\mathcal{RMS}(g)}\dim\left(W\right)\\
	&=\rind(f)+\rind(g).
	\end{align*}
	we complete the proof. 
\end{proof}	
The next corollary provides a sufficient condition on bent functions $f$ and $g$ for $f\oplus g$ being not in the $\MClass$ class in terms of their relaxed $\mathcal{M}$-subspaces.
\begin{corollary}\label{corollary: Sufficient condition for outside M}
	Let $f\colon\F_2^n\rightarrow\F_2$ and $g\colon\F_2^m\rightarrow\F_2$ be two Boolean bent functions. If $f$ and $g$ satisfy $\rind(f)<n/2$ and $\rind(g) \le m/2$, then $f\oplus k \cdot g\notin\MClass$ on $\F_2^{n+km}$ for all $k\in\N$.
\end{corollary}
\begin{remark}
	Throughout the paper we will call a Boolean function $f$ on $\F_2^n$ \emph{strongly extendable}, if $\rind(f)<n/2$ and \emph{weakly extendable}, if $\rind(f)=n/2$. In this way, if one wants to extend a strongly extendable function $f$ with Corollary~\ref{corollary: Sufficient condition for outside M}, it is enough to take a weakly extendable function $g$, while for the extension of a weakly extendable function $g$ one has to take a strongly extendable function $f$.
\end{remark}
\begin{remark}\label{remark: Compute rind(f)}
	For a given function $f$ one can compute the relaxed linearity index $\rind(f)$ in the same way as the linearity index $\ind(f)$, but with only one change. Instead of the second-order derivative $D_{\mathbf{a},\mathbf{b}}f$, given by its ANF
	$$D_{\mathbf{a},\mathbf{b}}f(\mathbf{x})=\bigoplus\limits_{\substack{\mathbf{v}\in\F_2^n}}c_{\mathbf{v}}({\mathbf{a},\mathbf{b})} \left( \prod_{i=1}^{n} x_i^{v_i} \right),$$
	where coefficients $c_\mathbf{v}$ depend on $\mathbf{a}$ and $\mathbf{b}$, one considers the \emph{``relaxed'' second-order derivative} $RD_{\mathbf{a},\mathbf{b}}f$, defined by $RD_{\mathbf{a},\mathbf{b}}f(\mathbf{x}):= D_{\mathbf{a},\mathbf{b}}f(\mathbf{x})\oplus c_\mathbf{0}(\mathbf{a},\mathbf{b})$ and use it as the input of Algorithm~\ref{algorithm: f in M} in the way already described in Remark~\ref{remark: Compute ind(f)}. 
\end{remark}
\subsection{Application to homogeneous cubic bent functions without affine derivatives}\label{subsection: Application}
In order to use Corollary~\ref{corollary: Sufficient condition for outside M} for the construction of cubic bent functions outside $\MClass$, which can be  homogeneous or have no affine derivatives, we need to find first such functions in a small number of variables and check, whether they are weakly or strongly extendable.

First we check, whether the equivalence classes of cubic bent functions in six~\cite[p. 303]{ROTHAUS1976300} and eight~\cite[p. 102]{Braeken2006} variables, contain functions with the mentioned properties. Since all cubic bent functions in 6 and 8 variables are members of the $\MClass$ class, as it was shown in~\cite[p. 37]{Dillon1972} and~\cite[p. 103]{Braeken2006} respectively, the best what one expects to find is a weakly extendable cubic bent function. In this way:
\begin{itemize}
	\item The only (up to equivalence) weakly extendable cubic bent function in 6 variables is the third Rothaus' function~\cite[p. 303]{ROTHAUS1976300}, denoted here by $R_3$. It has no affine derivatives and is not equivalent to any homogeneous cubic bent function.
	\item An example of  weakly extendable homogeneous cubic bent function in 8 variables is given by the function $h^8_1$. Like any other cubic bent function in eight variables, it has affine derivatives~\cite{HOU1998149}.
\end{itemize}
Now we analyze homogeneous cubic bent functions in 10 and 12 variables.
\begin{itemize}
	\item An example of a strongly extendable cubic bent function in 10 variables is represented by the function $h^{10}_4$, which is simultaneously homogeneous and has no affine derivatives.
	\item Since all the mentioned functions in 12 variables belong to the $\MClass$ class, they can not be strongly extendable. Nevertheless, among them we found a weakly extendable homogeneous function $h^{12}_5$ without affine derivatives.
\end{itemize}
We summarize these data in Table~\ref{table: Good Examples} and list all the used functions in the Appendix~\ref{section: Appendix}.
\begin{table}[H]
	\centering
	\caption{Extendable cubic bent functions in a small number of variables.}\label{table: Good Examples}
	\scalebox{1}{
		\begin{tabular}{|c|c|c|c|c|}
			\hline
			$\#$ of variables, $n$    & $6$ 				& $8$ 					& $10$ 			& $12$ 		\\ \hline
			$\rind$          	  & $3$ 				& $4$					& $4$  			& $6$ 		\\ \hline			
			Is homogeneous?        	  & $\mathbf{\times}$	& $\checkmark$  		& $\checkmark$  & $\checkmark$ 	\\ \hline
			Has no aff. derivatives?  & $\checkmark$   		& $\mathbf{\times}$    	& $\checkmark$  & $\checkmark$\\ \hline 
			Example			  & $R_3$				& $h^8_1$				& $h^{10}_4$	& $h^{12}_5$			\\ \hline     			
		\end{tabular}
	}
\end{table}
Now we proceed to the proof of our main theorem: the series of existence results about cubic bent functions with nice cryptographic properties.
\begin{theorem}\label{theorem: Existence 1} On $\F_2^n$ there exist:
	\begin{enumerate}
	\item Cubic bent functions outside $\MClass$ for all $n\ge 10$.
	\item Cubic bent functions without affine derivatives outside $\MClass$ for all $n\ge 26$.
	\item Homogeneous cubic bent functions outside $\MClass$ for all $n\ge 26$.
	\item Homogeneous cubic bent functions without affine derivatives outside $\MClass$ for all $n\ge 50$.
\end{enumerate}
\end{theorem}
\begin{proof}
In all the four cases the idea of the proof is the same: construct a strongly extendable Boolean function $h_n$ in $n=6i+8j+10k+12l$ variables of the form
\begin{equation}\label{equation: Main Construction}
h_n:=i\cdot R_3 \oplus j\cdot h^8_1 \oplus k\cdot h^{10}_4 \oplus l\cdot h^{12}_5
\end{equation}
and find the minimal value $n_0$, such that for all $n\ge n_0$ the function $h_n$ inherits the properties of its components from Table~\ref{table: Good Examples}. Since the only strongly extendable function is $h^{10}_4$ in $10$ variables, we require that in all the four cases below $k\ne 0$:

\noindent \emph{Case 1.} Since the first case has nothing to do with homogeneity and having no affine derivatives, one can use all the components from Table~\ref{table: Good Examples}. Clearly, the smallest value of $n$ is $n_0=16$ and in order to cover the missing values of $n\in \{12,14\}$, we construct a function $h_n'$ of the form
\begin{equation*}
h_{n}'(x_{1},\ldots,x_{n}):=h^{10}_4(x_{1},\ldots,x_{10})\oplus Q_k(x_{11},\ldots,x_{n})\mbox{ with } k=n-10.
\end{equation*}
Here $Q_{k}:=f_{id,0}$ is the quadratic bent function in $k$ variables, defined by the ``standard'' inner product on $\F_2^{k}$. Since for the quadratic bent function $Q_{k}$ its relaxed linearity index $\rind(Q_{k})=k$, we can not use Corollary~\ref{corollary: Sufficient condition for outside M}. However, by the second part of Theorem~\ref{theorem: MM Subspaces of Direct Sums},  one can verify, that $h_n'\notin\MClass$, by showing, that none of the vector subspaces $U$ of the form $$\{U\subseteq V \times W\colon V\in\mathcal{RMS}(h^{10}_4),W\in\mathcal{RMS}(Q_k)\}$$ is an $\mathcal{M}$-subspace of the function $h_n'$.

\noindent \emph{Case 2.} 
Since there are no weakly extendable homogeneous cubic bent functions in six variables, we can use only components $h^{8}_1,h^{10}_4,h^{12}_5$ in the equation~\eqref{equation: Main Construction}. One can see, that the smallest value of $n$ is $n_0=26$ and the missing values are in the set $\{14,16,24\}$.

\noindent \emph{Case 3.} First, we observe that the direct sum of two functions has no affine derivatives, if and only if both of them have no affine derivatives. Hence, the only functions we can use are $R_3,h^{10}_4,h^{12}_5$. In this way, the smallest value of $n$ is $n_0=26$ and the missing values are in the set $\{12,14,18,24\}$.

\noindent \emph{Case 4.} Finally, since the only extendable functions, which are simultaneously homogeneous and have no affine derivatives are $h^{10}_4$ and $h^{12}_5$, we observe, that the smallest value of $n$ is $n_0=50$ and the missing values of $n$ are in the set $\{12,14,16,18,24,26,28,36,38,48\}$, which completes the proof.
\end{proof}

\section{Conclusion}\label{section: Conclusion}
In this paper we proved the existence of cubic bent functions outside the completed Maiorana-McFarland class $\MClass$ on $\F_2^n$ for all $n\ge 10$ and showed that for almost all values of $n$ these functions can simultaneously be homogeneous and have no affine derivatives. The reason, why some values of $n$ are not covered by our proof is explained by the non-existence of examples with desired properties in 6 and 8 variables, which are necessary for the used recursive framework. 

In general, we expect that homogeneous cubic bent functions without affine derivatives outside $\MClass$ exist for all even $n\ge10$ and we leave this as an open problem. Since our proof technique is based on the direct sum construction of functions, some of them being members of $\MClass$, the functions constructed in such a way will presumably have bad cryptographic primitives (see~\cite[p. 330]{carlet_2010}). Thus, we suggest the following problem.
\begin{openproblem}
	Construct homogeneous cubic bent functions without affine derivatives outside the $\MClass$ class without the use of the direct sum.
\end{openproblem}
The next problem, which we would like to address, is related to the normality of cubic bent functions. Recall that a Boolean function $f$ on $\F_2^n$ is said to be normal (weakly normal), when it is constant (affine, but not constant) respectively, on some affine subspace $U$ of $\F_2^n$ of dimension $\lceil n / 2\rceil$. In this case $f$ is said to be normal (weakly normal) with respect to the flat $U$.  It is well-known that all quadratic bent functions are normal. Moreover, one can also construct non-normal as well as non-weakly normal bent functions of all degrees $d\ge4$, as it follows from~\cite[Fact 22]{CanteautDDL06}. At the same time all cubic bent functions in $n=6$ variables are normal or weakly-normal, while for $n=8$ they are proved to be normal~\cite{Charpin04}.

Since the functions $h^{10}_3$ and $h^{10}_4$ do not belong to the completed Maiorana-McFarland class, they are good candidates to be checked for the normality. Based on our parallel implementation of~\cite[Algorithm 1]{CanteautDDL06} in \texttt{Mathematica}~\cite{WolframMathematica112} we observe, that the function $h^{10}_3$ is normal on the flat $48\oplus \langle \mbox{g3},\mbox{8p},\mbox{4q},\mbox{2m},\mbox{1j} \rangle$ and the function $h^{10}_4$ is normal on the flat $5\oplus \langle \mbox{i5},\mbox{8h},\mbox{6n},\mbox{1g},\mbox{f}\rangle$. Here we describe each binary vector of a flat by 32-base representation, using the following alphabet  
\begin{equation}\label{equation: alphabet}
0\mapsto0,\ldots,\mbox{f}\mapsto15,\mbox{g}\mapsto16,\ldots,\mbox{v}\mapsto31.
\end{equation}
In this way, since one still has no examples of non-weakly normal cubic bent functions, it is reasonable to ask the following question.
\begin{openproblem}
	Do non-weakly normal cubic bent functions exist?
\end{openproblem}

\noindent Finally we list all the homogeneous cubic bent functions used in the paper.

\section*{Acknowledgments}
	The authors would like to thank Pantelimon St\u{a}nic\u{a} for providing homogeneous cubic bent functions from~\cite[p. 149]{Charnes2002}.

\bibliographystyle{spmpsci}

\appendix
\section{Appendix: Known inequivalent homogeneous cubic bent functions}
\label{section: Appendix}
Algebraic normal forms of $n$-variable homogeneous cubic bent functions used in the paper. We abbreviated $0\le i\le9$ for the variable $x_i$, variables $x_{10}$ and $x_{11}$ are replaced by $a$ and $b$ respectively.

\begin{itemize}
	\item[$h^{6}_{1}.$] $012 \oplus 013 \oplus 014 \oplus 023 \oplus 025 \oplus 034 \oplus 035 \oplus 045 \oplus 124 \oplus 125 \oplus 134 \oplus 135 \oplus 145 \oplus 234 \oplus 235 \oplus 245$
	\item[$h^{8}_{1}.$] $014 \oplus 016 \oplus 023 \oplus 025 \oplus 026 \oplus 027 \oplus 037 \oplus 045 \oplus 046 \oplus 047 \oplus 067 \oplus 123 \oplus 126 \oplus 135 \oplus 147 \oplus 157 \oplus 235 \oplus 236 \oplus 245 \oplus 246 \oplus 257 \oplus 346 \oplus 347 \oplus 356 \oplus 357 \oplus 367 \oplus 456 \oplus 457$
	\item[$h^{8}_{2}.$] $012 \oplus 013 \oplus 015 \oplus 016 \oplus 023 \oplus 035 \oplus 037 \oplus 046 \oplus 047 \oplus 125 \oplus 136 \oplus 145 \oplus 146 \oplus 156 \oplus 237 \oplus 245 \oplus 247 \oplus 256 \oplus 257 \oplus 267 \oplus 346 \oplus 347 \oplus 357 \oplus 467$
	\item[$h^{10}_{1}.$] $015 \oplus 017 \oplus 018 \oplus 019 \oplus 023 \oplus 026 \oplus 027 \oplus 028 \oplus 034 \oplus 038 \oplus 039 \oplus 046 \oplus 048 \oplus 049 \oplus 067 \oplus 068 \oplus 125 \oplus 126 \oplus 128 \oplus 129 \oplus 159 \oplus 168 \oplus 178 \oplus 179 \oplus 189 \oplus 236 \oplus 239 \oplus 245 \oplus 246 \oplus 247 \oplus 248 \oplus 256 \oplus 258 \oplus 259 \oplus 269 \oplus 279 \oplus 345 \oplus 346 \oplus 356 \oplus 357 \oplus 359 \oplus 367 \oplus 378 \oplus 379 \oplus 389 \oplus 457 \oplus 459 \oplus 467 \oplus 468 \oplus 479 \oplus 589 \oplus 678 \oplus 679$
	\item[$h^{10}_{2}.$] $012 \oplus 013 \oplus 014 \oplus 015 \oplus 016 \oplus 017 \oplus 018 \oplus 019 \oplus 023 \oplus 024 \oplus 025 \oplus 029 \oplus 036 \oplus 037 \oplus 038 \oplus 045 \oplus 048 \oplus 049 \oplus 056 \oplus 059 \oplus 067 \oplus 068 \oplus 078 \oplus 079 \oplus 123 \oplus 126 \oplus 127 \oplus 128 \oplus 134 \oplus 135 \oplus 139 \oplus 145 \oplus 148 \oplus 149 \oplus 156 \oplus 159 \oplus 167 \oplus 168 \oplus 178 \oplus 179 \oplus 234 \oplus 235 \oplus 236 \oplus 237 \oplus 238 \oplus 239 \oplus 245 \oplus 248 \oplus 249 \oplus 256 \oplus 259 \oplus 267 \oplus 268 \oplus 278 \oplus 279 \oplus 345 \oplus 348 \oplus 349 \oplus 356 \oplus 359 \oplus 367 \oplus 368 \oplus 378 \oplus 379 \oplus 456 \oplus 478 \oplus 489 \oplus 568 \oplus 579 \oplus 679$
	\item[$h^{10}_{3}.$] $012 \oplus 015 \oplus 017 \oplus 019 \oplus 024 \oplus 025 \oplus 028 \oplus 029 \oplus 034 \oplus 039 \oplus 046 \oplus 049 \oplus 058 \oplus 067 \oplus 078 \oplus 089 \oplus 125 \oplus 126 \oplus 128 \oplus 129 \oplus 159 \oplus 168 \oplus 178 \oplus 179 \oplus 189 \oplus 236 \oplus 239 \oplus 245 \oplus 246 \oplus 247 \oplus 248 \oplus 256 \oplus 258 \oplus 259 \oplus 269 \oplus 279 \oplus 345 \oplus 346 \oplus 356 \oplus 357 \oplus 359 \oplus 367 \oplus 378 \oplus 379 \oplus 389 \oplus 457 \oplus 459 \oplus 467 \oplus 468 \oplus 479 \oplus 589 \oplus 678 \oplus 679$
	\item[$h^{10}_{4}.$] $015 \oplus 016 \oplus 017 \oplus 019 \oplus 023 \oplus 024 \oplus 026 \oplus 028 \oplus 029 \oplus 034 \oplus 035 \oplus 037 \oplus 038 \oplus 039 \oplus 046 \oplus 056 \oplus 057 \oplus 059 \oplus 068 \oplus 069 \oplus 089 \oplus 124 \oplus 127 \oplus 128 \oplus 129 \oplus 135 \oplus 136 \oplus 137 \oplus 145 \oplus 148 \oplus 156 \oplus 158 \oplus 159 \oplus 167 \oplus 169 \oplus 178 \oplus 179 \oplus 189 \oplus 236 \oplus 238 \oplus 245 \oplus 246 \oplus 247 \oplus 249 \oplus 257 \oplus 258 \oplus 269 \oplus 278 \oplus 279 \oplus 289 \oplus 346 \oplus 348 \oplus 349 \oplus 357 \oplus 359 \oplus 367 \oplus 368 \oplus 369 \oplus 379 \oplus 389 \oplus 457 \oplus 458 \oplus 459 \oplus 468 \oplus 469 \oplus 478 \oplus 479 \oplus 489 \oplus 567 \oplus 579 \oplus 589 \oplus 679$
	\item[$h^{12}_{1}.$] $024 \oplus 027 \oplus 02\text{a} \oplus 02\text{b} \oplus 034 \oplus 038 \oplus 046 \oplus 049 \oplus 056 \oplus 05\text{a} \oplus 068 \oplus 06\text{b} \oplus 078 \oplus 08\text{a} \oplus 09\text{a} \oplus 123 \oplus 127 \oplus 135 \oplus 138 \oplus 13\text{b} \oplus 145 \oplus 149 \oplus 157 \oplus 15\text{a} \oplus 167 \oplus 16\text{b} \oplus 179 \oplus 189 \oplus 19\text{b} \oplus 1\text{a}\text{b} \oplus 235 \oplus 236 \oplus 237 \oplus 24\text{b} \oplus 25\text{b} \oplus 26\text{b} \oplus 278 \oplus 289 \oplus 29\text{a} \oplus 346 \oplus 347 \oplus 348 \oplus 389 \oplus 39\text{a} \oplus 3\text{a}\text{b} \oplus 457 \oplus 458 \oplus 459 \oplus 49\text{a} \oplus 4\text{a}\text{b} \oplus 568 \oplus 569 \oplus 56\text{a} \oplus 5\text{a}\text{b} \oplus 679 \oplus 67\text{a} \oplus 67\text{b} \oplus 78\text{a} \oplus 78\text{b} \oplus 89\text{b}$
	\item[$h^{12}_{2}.$] $024 \oplus 025 \oplus 027 \oplus 029 \oplus 02\text{a} \oplus 02\text{b} \oplus 034 \oplus 036 \oplus 038 \oplus 03\text{a} \oplus 046 \oplus 047 \oplus 049 \oplus 04\text{b} \oplus 056 \oplus 058 \oplus 05\text{a} \oplus 068 \oplus 069 \oplus 06\text{b} \oplus 078 \oplus 07\text{a} \oplus 08\text{a} \oplus 08\text{b} \oplus 09\text{a} \oplus 123 \oplus 125 \oplus 127 \oplus 129 \oplus 135 \oplus 136 \oplus 138 \oplus 13\text{a} \oplus 13\text{b} \oplus 145 \oplus 147 \oplus 149 \oplus 14\text{b} \oplus 157 \oplus 158 \oplus 15\text{a} \oplus 167 \oplus 169 \oplus 16\text{b} \oplus 179 \oplus 17\text{a} \oplus 189 \oplus 18\text{b} \oplus 19\text{b} \oplus 1\text{a}\text{b} \oplus 237 \oplus 239 \oplus 23\text{a} \oplus 245 \oplus 247 \oplus 249 \oplus 256 \oplus 257 \oplus 25\text{a} \oplus 26\text{b} \oplus 278 \oplus 279 \oplus 27\text{a} \oplus 28\text{b} \oplus 29\text{b} \oplus 348 \oplus 34\text{a} \oplus 34\text{b} \oplus 356 \oplus 358 \oplus 35\text{a} \oplus 367 \oplus 368 \oplus 36\text{b} \oplus 389 \oplus 38\text{a} \oplus 38\text{b} \oplus 459 \oplus 45\text{b} \oplus 467 \oplus 469 \oplus 46\text{b} \oplus 478 \oplus 479 \oplus 49\text{a} \oplus 49\text{b} \oplus 56\text{a} \oplus 578 \oplus 57\text{a} \oplus 589 \oplus 58\text{a} \oplus 5\text{a}\text{b} \oplus 67\text{b} \oplus 689 \oplus 68\text{b} \oplus 69\text{a} \oplus 69\text{b} \oplus 79\text{a} \oplus 7\text{a}\text{b} \oplus 8\text{a}\text{b}$
	\item[$h^{12}_{3}.$] $023 \oplus 024 \oplus 026 \oplus 027 \oplus 028 \oplus 02\text{a} \oplus 035 \oplus 038 \oplus 03\text{b} \oplus 045 \oplus 046 \oplus 048 \oplus 049 \oplus 04\text{a} \oplus 057 \oplus 05\text{a} \oplus 067 \oplus 068 \oplus 06\text{a} \oplus 06\text{b} \oplus 079 \oplus 089 \oplus 08\text{a} \oplus 09\text{b} \oplus 0\text{a}\text{b} \oplus 124 \oplus 127 \oplus 12\text{a} \oplus 12\text{b} \oplus 134 \oplus 135 \oplus 137 \oplus 138 \oplus 139 \oplus 13\text{b} \oplus 146 \oplus 149 \oplus 156 \oplus 157 \oplus 159 \oplus 15\text{a} \oplus 15\text{b} \oplus 168 \oplus 16\text{b} \oplus 178 \oplus 179 \oplus 17\text{b} \oplus 18\text{a} \oplus 19\text{a} \oplus 19\text{b} \oplus 234 \oplus 235 \oplus 237 \oplus 239 \oplus 23\text{b} \oplus 246 \oplus 24\text{a} \oplus 24\text{b} \oplus 256 \oplus 26\text{b} \oplus 278 \oplus 28\text{a} \oplus 28\text{b} \oplus 29\text{a} \oplus 2\text{a}\text{b} \oplus 345 \oplus 346 \oplus 348 \oplus 34\text{a} \oplus 357 \oplus 35\text{b} \oplus 367 \oplus 389 \oplus 39\text{b} \oplus 3\text{a}\text{b} \oplus 456 \oplus 457 \oplus 459 \oplus 45\text{b} \oplus 468 \oplus 478 \oplus 49\text{a} \oplus 567 \oplus 568 \oplus 56\text{a} \oplus 579 \oplus 589 \oplus 5\text{a}\text{b} \oplus 678 \oplus 679 \oplus 67\text{b} \oplus 68\text{a} \oplus 69\text{a} \oplus 789 \oplus 78\text{a} \oplus 79\text{b} \oplus 7\text{a}\text{b} \oplus 89\text{a} \oplus 89\text{b} \oplus 9\text{a}\text{b}$
	\item[$h^{12}_{4}.$] $023 \oplus 025 \oplus 026 \oplus 027 \oplus 028 \oplus 029 \oplus 036 \oplus 038 \oplus 03\text{a} \oplus 045 \oplus 047 \oplus 048 \oplus 049 \oplus 04\text{a} \oplus 04\text{b} \oplus 058 \oplus 05\text{a} \oplus 067 \oplus 069 \oplus 06\text{a} \oplus 06\text{b} \oplus 07\text{a} \oplus 089 \oplus 08\text{b} \oplus 0\text{a}\text{b} \oplus 125 \oplus 127 \oplus 129 \oplus 12\text{b} \oplus 134 \oplus 136 \oplus 137 \oplus 138 \oplus 139 \oplus 13\text{a} \oplus 147 \oplus 149 \oplus 14\text{b} \oplus 156 \oplus 158 \oplus 159 \oplus 15\text{a} \oplus 15\text{b} \oplus 169 \oplus 16\text{b} \oplus 178 \oplus 17\text{a} \oplus 17\text{b} \oplus 18\text{b} \oplus 19\text{a} \oplus 234 \oplus 235 \oplus 237 \oplus 239 \oplus 23\text{b} \oplus 246 \oplus 24\text{a} \oplus 24\text{b} \oplus 256 \oplus 26\text{b} \oplus 278 \oplus 28\text{a} \oplus 28\text{b} \oplus 29\text{a} \oplus 2\text{a}\text{b} \oplus 345 \oplus 346 \oplus 348 \oplus 34\text{a} \oplus 357 \oplus 35\text{b} \oplus 367 \oplus 389 \oplus 39\text{b} \oplus 3\text{a}\text{b} \oplus 456 \oplus 457 \oplus 459 \oplus 45\text{b} \oplus 468 \oplus 478 \oplus 49\text{a} \oplus 567 \oplus 568 \oplus 56\text{a} \oplus 579 \oplus 589 \oplus 5\text{a}\text{b} \oplus 678 \oplus 679 \oplus 67\text{b} \oplus 68\text{a} \oplus 69\text{a} \oplus 789 \oplus 78\text{a} \oplus 79\text{b} \oplus 7\text{a}\text{b} \oplus 89\text{a} \oplus 89\text{b} \oplus 9\text{a}\text{b}$
	\item[$h^{12}_{5}.$] $024 \oplus 025 \oplus 027 \oplus 02\text{a} \oplus 038 \oplus 03\text{a} \oplus 046 \oplus 047 \oplus 049 \oplus 05\text{a} \oplus 068 \oplus 069 \oplus 06\text{b} \oplus 08\text{a} \oplus 08\text{b} \oplus 127 \oplus 129 \oplus 135 \oplus 136 \oplus 138 \oplus 13\text{b} \oplus 149 \oplus 14\text{b} \oplus 157 \oplus 158 \oplus 15\text{a} \oplus 16\text{b} \oplus 179 \oplus 17\text{a} \oplus 19\text{b} \oplus 234 \oplus 235 \oplus 239 \oplus 23\text{a} \oplus 23\text{b} \oplus 245 \oplus 247 \oplus 249 \oplus 24\text{b} \oplus 256 \oplus 257 \oplus 25\text{a} \oplus 279 \oplus 27\text{a} \oplus 28\text{b} \oplus 29\text{a} \oplus 29\text{b} \oplus 2\text{a}\text{b} \oplus 345 \oplus 346 \oplus 34\text{a} \oplus 34\text{b} \oplus 356 \oplus 358 \oplus 35\text{a} \oplus 367 \oplus 368 \oplus 36\text{b} \oplus 38\text{a} \oplus 38\text{b} \oplus 3\text{a}\text{b} \oplus 456 \oplus 457 \oplus 45\text{b} \oplus 467 \oplus 469 \oplus 46\text{b} \oplus 478 \oplus 479 \oplus 49\text{b} \oplus 567 \oplus 568 \oplus 578 \oplus 57\text{a} \oplus 589 \oplus 58\text{a} \oplus 678 \oplus 679 \oplus 689 \oplus 68\text{b} \oplus 69\text{a} \oplus 69\text{b} \oplus 789 \oplus 78\text{a} \oplus 79\text{a} \oplus 7\text{a}\text{b} \oplus 89\text{a} \oplus 89\text{b} \oplus 8\text{a}\text{b} \oplus 9\text{a}\text{b}$
	\item[$h^{12}_{6}.$] $027 \oplus 029 \oplus 02\text{a} \oplus 02\text{b} \oplus 037 \oplus 038 \oplus 03\text{a} \oplus 03\text{b} \oplus 047 \oplus 048 \oplus 049 \oplus 04\text{b} \oplus 057 \oplus 058 \oplus 059 \oplus 05\text{a} \oplus 068 \oplus 069 \oplus 06\text{a} \oplus 06\text{b} \oplus 078 \oplus 07\text{b} \oplus 089 \oplus 09\text{a} \oplus 0\text{a}\text{b} \oplus 123 \oplus 126 \oplus 127 \oplus 128 \oplus 129 \oplus 12\text{a} \oplus 134 \oplus 138 \oplus 139 \oplus 13\text{a} \oplus 13\text{b} \oplus 145 \oplus 147 \oplus 149 \oplus 14\text{a} \oplus 14\text{b} \oplus 156 \oplus 157 \oplus 158 \oplus 15\text{a} \oplus 15\text{b} \oplus 167 \oplus 168 \oplus 169 \oplus 16\text{b} \oplus 237 \oplus 23\text{a} \oplus 23\text{b} \oplus 248 \oplus 24\text{a} \oplus 258 \oplus 25\text{b} \oplus 269 \oplus 26\text{a} \oplus 26\text{b} \oplus 278 \oplus 289 \oplus 28\text{b} \oplus 29\text{a} \oplus 29\text{b} \oplus 347 \oplus 348 \oplus 34\text{b} \oplus 359 \oplus 35\text{b} \oplus 367 \oplus 369 \oplus 379 \oplus 37\text{a} \oplus 389 \oplus 39\text{a} \oplus 3\text{a}\text{b} \oplus 456 \oplus 457 \oplus 458 \oplus 459 \oplus 45\text{b} \oplus 467 \oplus 47\text{b} \oplus 48\text{a} \oplus 48\text{b} \oplus 49\text{a} \oplus 568 \oplus 56\text{a} \oplus 578 \oplus 579 \oplus 57\text{b} \oplus 5\text{a}\text{b} \oplus 678 \oplus 67\text{a} \oplus 67\text{b} \oplus 689 \oplus 68\text{a} \oplus 69\text{a} \oplus 9\text{a}\text{b}$
	\item[$h^{12}_{7}.$] $027 \oplus 029 \oplus 02\text{a} \oplus 02\text{b} \oplus 037 \oplus 038 \oplus 03\text{a} \oplus 03\text{b} \oplus 047 \oplus 048 \oplus 049 \oplus 04\text{b} \oplus 057 \oplus 058 \oplus 059 \oplus 05\text{a} \oplus 068 \oplus 069 \oplus 06\text{a} \oplus 06\text{b} \oplus 078 \oplus 07\text{b} \oplus 089 \oplus 09\text{a} \oplus 0\text{a}\text{b} \oplus 123 \oplus 126 \oplus 127 \oplus 128 \oplus 129 \oplus 12\text{a} \oplus 134 \oplus 138 \oplus 139 \oplus 13\text{a} \oplus 13\text{b} \oplus 145 \oplus 147 \oplus 149 \oplus 14\text{a} \oplus 14\text{b} \oplus 156 \oplus 157 \oplus 158 \oplus 15\text{a} \oplus 15\text{b} \oplus 167 \oplus 168 \oplus 169 \oplus 16\text{b} \oplus 237 \oplus 23\text{a} \oplus 23\text{b} \oplus 248 \oplus 24\text{a} \oplus 258 \oplus 25\text{b} \oplus 269 \oplus 26\text{a} \oplus 26\text{b} \oplus 278 \oplus 289 \oplus 28\text{b} \oplus 29\text{a} \oplus 29\text{b} \oplus 347 \oplus 348 \oplus 34\text{b} \oplus 356 \oplus 359 \oplus 367 \oplus 369 \oplus 36\text{a} \oplus 379 \oplus 37\text{a} \oplus 389 \oplus 39\text{a} \oplus 457 \oplus 458 \oplus 459 \oplus 467 \oplus 46\text{a} \oplus 47\text{b} \oplus 48\text{a} \oplus 48\text{b} \oplus 49\text{a} \oplus 4\text{a}\text{b} \oplus 569 \oplus 56\text{a} \oplus 578 \oplus 579 \oplus 57\text{b} \oplus 58\text{b} \oplus 59\text{b} \oplus 5\text{a}\text{b} \oplus 678 \oplus 67\text{a} \oplus 67\text{b} \oplus 689 \oplus 8\text{a}\text{b}$
\end{itemize}

\begin{table}[H]
	\centering
	\caption{The known homogeneous cubic bent functions in a small number of variables and their invariants. Functions $h^6_1$ and $h^8_1,h^8_2$ describe up to equivalence all homogeneous functions in 6 and 8 variables, respectively. Functions $h^{10}_3$ and $h^{10}_1$ are the first and the second 10-variable functions from~\cite[p. 15]{MengNovel2004}. Functions $h^{10}_2$ and $h^{10}_4$ are representatives of equivalence classes of functions, constructed in~\cite[p. 149]{Charnes2002}. Functions $h^{12}_i$ for $1\le i\le5$ are representatives of equivalence classes of functions, constructed in~\cite[p. 149]{Charnes2002}. Functions $h^{12}_6$ and $h^{12}_7$ were constructed in Subsection~\ref{subsection: New from old}.} 
	\label{table: Functions from the paper}
	\begin{subtable}[t]{.5\linewidth}
		\centering
		\scalebox{1}{
			\begin{tabular}{|c|c|c|c|}		
				\hline
				$h^{n}_{i}$     & $\ind(h^{n}_{i})$ & $\rind(h^{n}_{i})$ & $\dim(\mathbb{FP}_{h^{n}_{i}})$   \\ \hline 
				$h^{6}_{1}$     & 3 & 4 & 3  \\ \hline
				$h^{8}_{1}$  &   4    &    \textbf{4}    & 1 \\ \hline
				$h^{8}_{2}$  &  4     &     5   & 2 \\ \hline				
				$h^{10}_{1}$  &   5    &   5     & 1 \\ \hline                           
				$h^{10}_{2}$  &   5   &   5     & 1 \\ \hline
				$h^{10}_{3}$  &   \textbf{4}    &     \textbf{4}  & 1 \\ \hline
				$h^{10}_{4}$  &   \textbf{2}    &    \textbf{4}    & \textbf{0} \\ \hline
		\end{tabular}}
	\end{subtable}%
	\begin{subtable}[t]{.5\linewidth}
		\centering
		\scalebox{1}{
			\begin{tabular}{|c|c|c|c|}		
				\hline
				$h^{n}_{i}$     & $\ind(h^{n}_{i})$ & $\rind(h^{n}_{i})$ & $\dim(\mathbb{FP}_{h^{n}_{i}})$   \\ \hline	
				$h^{12}_{1}$  & 6      &     \textbf{6}   & 2  \\ \hline                           
				$h^{12}_{2}$  & 6      &     \textbf{6}   & 2 \\ \hline
				$h^{12}_{3}$  & 6      &     7   & 1  \\ \hline
				$h^{12}_{4}$  & 6      &     7   & 2  \\ \hline
				$h^{12}_{5}$  & 6      &     \textbf{6}   & \textbf{0}  \\ \hline
				$h^{12}_{6}$  & 6      &     $\ge$7   & 1  \\ \hline
				$h^{12}_{7}$  & 6      &    $\ge$7   & 1  \\ \hline
		\end{tabular}}
	\end{subtable}%
\end{table}

For each homogeneous cubic bent function $h^{n}_i\in\MClass$ on $\F_2^n$ we list the collection $\mathcal{M}_{n/2}(h^{n}_i)$ as a $|\mathcal{M}_{n/2}(h^{n}_i)| \times n/2$ matrix in the following way. Each row of  $\mathcal{M}_{n/2}(h^{n}_i)$ describes the Gauss-Jordan basis of an $\mathcal{M}$-subspace of $h^{n}_i$. Each element of a basis is given by 32-base number, which can be converted to the binary vector of length $n$, using the alphabet~\eqref{equation: alphabet}. For instance, using this conversion one can check, that the first row of the matrix $\mathcal{MS}_6(h^{12}_3)$ describes the $\gjb(U)$ of the $\mathcal{M}$-subspace $U$, given in~\eqref{equation: MSubspace from the Observation}.

\begin{itemize}
	\item $\mathcal{MS}_5(h^{10}_1)=\scalebox{0.95}{$\left(
		\begin{array}{ccccc}
			\text{o2} & \text{4l} & \text{2m} & \text{1j} & \text{f} \\
		\end{array}
		\right)$}$, $\mathcal{MS}_5(h^{10}_2)=\scalebox{0.95}{$\left(
		\begin{array}{ccccc}
			\text{o0} & 60 & 12 & \text{o} & 5 \\
		\end{array}
		\right)$}$;\ \\[0.5mm]
	\item $\mathcal{MS}_6(h^{12}_1)=\scalebox{0.95}{$\left(
		\begin{array}{cccccc}
			\text{22r} & \text{10m} & \text{it} & \text{8e} & 66 & 17 \\
			\text{20q} & \text{12o} & \text{in} & \text{af} & \text{4s} & \text{1p} \\
			\text{21c} & \text{10d} & \text{gs} & \text{9n} & \text{5r} & \text{2e} \\
			\text{20b} & \text{11u} & \text{gj} & \text{9t} & 47 & 33 \\
			\text{20v} & \text{11k} & \text{hh} & \text{9o} & \text{5f} & \text{3i}
		\end{array}\right)$}, \mathcal{MS}_6(h^{12}_3)=\scalebox{0.95}{$\left(
		\begin{array}{cccccc}
			300 & \text{gg} & 88 & 44 & 22 & 11 \\
			\text{21u} & \text{10v} & \text{hh} & 99 & 55 & 33 \\
			\text{20v} & \text{11u} & \text{hh} & 99 & 55 & 33
		\end{array}
		\right)$},$\ \\[0.5mm]
	\item[] $\mathcal{MS}_6(h^{12}_5)=\scalebox{0.95}{$\left(
		\begin{array}{cccccc} 300 & \text{gg} & 88 & 44 & 22 & 11\end{array}
		\right)$}$, $\mathcal{MS}_6(h^{12}_1)=\mathcal{MS}_6(h^{12}_2)$, $\mathcal{MS}_6(h^{12}_3)=$
	
	$\mathcal{MS}_6(h^{12}_4)=\mathcal{MS}_6(h^{12}_6)=\mathcal{MS}_6(h^{12}_7)$.
\end{itemize} 
\end{document}